\documentclass[11pt]{article}
\pdfoutput=1

\usepackage[T1]{fontenc}
\usepackage[utf8]{inputenc}
\usepackage{lmodern}
\usepackage{amsmath,amssymb,amsthm,mathtools}
\usepackage{graphicx}
\usepackage{tikz}
\usetikzlibrary{arrows.meta,positioning,calc,decorations.pathmorphing,fit,backgrounds}
\usepackage[margin=1in]{geometry}
\usepackage{microtype}
\usepackage[round,authoryear]{natbib}
\usepackage{caption}
\usepackage[colorlinks=true,linkcolor=black,citecolor=black,urlcolor=blue]{hyperref}
\usepackage[capitalise,noabbrev]{cleveref}

\renewcommand{\emph}[1]{#1}

\theoremstyle{plain}
\newtheorem{theorem}{Theorem}
\newtheorem{lemma}{Lemma}
\newtheorem{proposition}{Proposition}
\newtheorem{corollary}{Corollary}

\theoremstyle{definition}
\newtheorem{definition}{Definition}
\theoremstyle{remark}
\newtheorem{remark}{Remark}

\newcommand{\Pt}{P}
\newcommand{\Amat}{A}
\newcommand{\Rt}{R(t)}
\newcommand{\norm}[1]{\left\lVert #1 \right\rVert}
\newcommand{\ncal}{\widetilde{N}}
\newcommand{\real}{\mathbb{R}}
\newcommand{\ket}{\ker \Amat}
\DeclareMathOperator{\rank}{rank}
\DeclareMathOperator{\diag}{diag}
\DeclareMathOperator{\range}{range}
\DeclareMathOperator{\relint}{relint}

\title{Mathematical modelling of a multicluster open system: transfer-operator
recovery from aggregate data under a least-action prior}
\author{Artem P. Nevecheria\thanks{Kuban State University, Krasnodar, 350040, Russia.
\texttt{artiom.nevecherya@gmail.com}. ORCID: 0000-0001-6736-4691.}}
\date{}

\begin{document}
\maketitle

\begin{abstract}
\noindent
Many open systems are observed only through aggregate snapshots, while the rule that moves
objects between states stays hidden. This paper studies the inverse problem of recovering the
transfer operator of a multicluster open system from two adjacent snapshots. The consistent
operators fill an affine manifold: the data fix the aggregate flows and leave the internal
routing free, so every recovery method commits to a prior that selects one operator. Tikhonov
regularisation with a zero reference gives the uniform outflow on each source simplex, which
underestimates retention in an inertial system. We propose least action on the operator
trajectory, under which the rule changes as little as the data permit. In discrete time this
is exact as a probability statement: the action prices a trajectory by its likelihood under a
Gaussian reference drawn toward full retention, so least action selects the most probable
trajectory consistent with the balance. A causal construction keeps the balance exact by
moving the operator only inside the null space of the constraints, and we relate it to
optimal transport, to Perron-Frobenius coarse-graining, and to the Schr\"odinger bridge. Two
instantiations carry it: the Russian labour market and a radiation-damped relativistic
electron ensemble. The construction holds the balance to machine precision where temporal
smoothing corrupts it by twelve percent on the labour series and sixty percent in the
physical one. The spectral gap is non-identifiable from aggregate snapshots but available
from micro-trajectories, where Ulam's method reproduces the Landau-Lifshitz rate to within
six percent. On the short labour series no method beats the naive forecast in point accuracy
at conventional significance levels. The value is structural: an interpretable operator, an
exactly preserved balance, and an explicit information limit of aggregate observation.
\end{abstract}



\section{Introduction}\label{sec:intro}

An open system that exchanges objects with its environment is often observed only through aggregate snapshots, the head counts of its states at successive instants.
The rule that moves objects between the states stays hidden, and the counts alone leave it under-determined, since the unknown transition probabilities outnumber the balance equations \citep{LeeJudgeZellner1970,BernsteinSheldon2016}.
The same problem recurs across domains.
It appears in flow econometrics, in epidemiology, in consumer-preference analysis, in single-cell biology, in origin-destination traffic, and in beam physics.
Classical instances give it a familiar shape.
One recovers an origin-destination matrix from link counts \citep{VanZuylen1980}, recovers gross labour flows from net ones \citep{BlanchardDiamond1990}, estimates a Markov chain from aggregate proportions \citep{Kalbfleisch1983}, or recovers source-destination traffic intensities from link data in network tomography \citep{Vardi1996}.

The recovery is an ill-posed inverse problem, and regularisation restores a unique solution \citep{Tikhonov1977,Morozov1987,Kabanikhin2012}.
The regularising term selects which solution is returned, so the choice of stabiliser is a prior on the operator.
The standard choices are the minimum-norm solution, the maximum-entropy solution \citep{Jaynes1957}, and the minimum distance to a target matrix \citep{VanZuylen1980}.
The regularisation parameter follows from the discrepancy principle, the L-curve \citep{Hansen1992}, or generalised cross-validation \citep{GolubHeathWahba1979}.

Recent work recovers the hidden dynamics from marginal snapshots by transport and by smoothing.
Optimal transport reconstructs cell-population trajectories from snapshots \citep{Schiebinger2019}, and its regularised unbalanced successors carry the construction to systems with growth \citep{ZhangLiZhou2025}.
Inverse optimal transport estimates a Markov chain from aggregate observations \citep{InverseOT2025}.
Temporal smoothing tracks a time-varying transition matrix by penalising its change between steps \citep{BrandShare2017}.
A parallel line learns transfer operators and their Perron-Frobenius spectra from data \citep{Froyland2014,BiSarrazinSchmitzer2026,IoannouKlusDosReis2025}, and Schrodinger-bridge estimation reads the most likely evolution between two distributions from finite samples \citep{PooladianNilesWeed2024,MaedaYaoNitanda2025}.
Robust estimators recover a transition matrix from ensembles of sample paths under model mismatch \citep{LeskeleDreveton2026}.
These methods fix the prior in advance and edit the data-consistent part of the operator.
The present paper differs there.

The data determine only a manifold of operators, and every point of it reproduces the observed snapshots exactly.
Selecting one operator is selecting a representative of that manifold, and the selection is made by a pair, a reference point and a geometry that measures distance to it.
This paper proposes least action on the operator trajectory as the reference and geometry for inertial systems.
Informally, the system keeps doing what it did, a reading we call behavioural laziness, after which the accompanying code release is named \citep{NevecheryaCode2026}.
The prior asks the rule of evolution to change as little as the data allow.
The construction imposes it exactly on the data-consistent manifold and without using future observations, so it moves only the unidentified part of the operator and holds the observed balance in place.

The contributions follow from this position.
The geometry of the stabiliser becomes an explicit hypothesis about the system, tied to a non-identifiability result that measures the freedom the data leave, since the solution manifold has dimension $2n^2+n$ (\cref{lem:rank,thm:nonident}).
The prior separates two kinds of inertia that earlier smoothing conflates, the inertia of the state and the inertia of the rule, and only the second reaches the operator and carries a signal the naive forecast cannot supply.
A causal null-space construction imposes the prior while preserving the observed balance exactly, so the recovered operator stays balance-exact to machine precision.
The connections to optimal transport and to the Perron-Frobenius coarse-graining of a transfer operator are stated and proved as formal results, and the construction is related to the Schrodinger bridge (\cref{sec:action}).
The least-action name is exact in discrete time: the two penalty terms of the action are, up to an additive constant, the negative logarithm of a Gaussian density on the operator trajectory, so minimisation under the balance constraints picks the most probable trajectory of an inertial reference process drawn toward full retention (\cref{rem:onsager}).
The continuum reading through the Onsager-Machlup action holds at the level of the discretised action, and the same remark records where it stops.
The approach transfers across two instantiations of one model, a Russian labour market and a relativistic electron ensemble.
The same solver runs without modification, and the inertial prior can steer the operator's unobservable component onto the same dissipative attractor that the physical relaxation selects.
Because the reference is chosen, this is a controllability demonstration.
The point-forecast gain from the dynamic prior is small and not significant on the short series available.
The value of the construction is structural, in what the prior fixes about the recovered operator, in the exact preservation of balance, and in the geometry it makes explicit.

\Cref{sec:model} builds the multicluster open system and its balance model, and \cref{sec:vector} casts the recovery as the segmentation problem, a Tikhonov formulation of an under-determined linear system.
\Cref{sec:nonident} proves the non-identifiability and reads the stabiliser as a choice of geometry on the solution manifold.
\Cref{sec:action} develops the least-action prior and its connections to optimal transport, the Schrodinger bridge, and the transfer operator.
\Cref{sec:num} gives the null-space smoothing and the temporal smoothing algorithms.
\Cref{sec:labour,sec:physics} report the two computational examples, and \cref{sec:applic} maps where the prior helps.
\Cref{sec:disc,sec:concl} discuss the limits and conclude.

\section{The multicluster open system}\label{sec:model}

\subsection{Informal setting}
We study an open system whose objects move between a fixed set of macro-states over discrete time.
The states are chosen before any data are seen, and their meaning depends on the application.
In the labour market an object is a worker and a state is an industry.
In the physical instantiation of \cref{sec:physics} an object is an electron and a state is a phase-space cell.
The system is open in the sense that objects enter it from an external reservoir and leave it through a departure channel.
At each instant we observe only the head counts of the states, and the rule that moves objects between them stays hidden.
This section builds the balance model that links two adjacent counts to that rule.

\subsection{States, subsystems and clusters}
The states of the system fall into two groups.
The first group holds the objects whose current state is known.
The second group holds the objects whose current state is unknown, indexed by the last known state where one exists.

\begin{definition}[Classifying states]\label{def:states}
A classifying state is one of $n$ predefined states $S_1^{(i)}$, $i=1,\dots,n$, such that every object occupies at most one classifying state at any instant.
The number $n$ and the meaning of each $S_1^{(i)}$ are fixed by the application.
\end{definition}

\begin{definition}[Subsystems]\label{def:subsystems}
Subsystem~1 consists of the objects that currently occupy a classifying state.
Subsystem~2 consists of the objects that currently occupy no classifying state.
An object of subsystem~2 that once belonged to subsystem~1 is assigned to the segment of the last classifying state $S_1^{(i)}$ it held.
The remaining objects of subsystem~2, those that have held no classifying state since entering the system, form the reserve segment $S_2^{(0)}$.
\end{definition}

A cluster in this work is a state taken together with the objects that occupy it at a given instant, so its population changes over time.
The clusters are predefined macro-states, indexed in advance by the modeller.
No similarity-clustering procedure is applied to the objects to form them.

\begin{definition}[Dynamic cluster]\label{def:cluster}
The dynamic cluster $S_j^{(i)}(t)$ is the state $S_j^{(i)}$ together with the set of objects occupying it at time $t$, where $j\in\{1,2\}$, $i_1=\overline{1,n}$ for the active clusters $S_1^{(i)}(t)$ and $i_2=\overline{0,n}$ for the passive clusters $S_2^{(i)}(t)$.
The active cluster $S_1^{(i)}(t)$ holds the objects of subsystem~1 that are in the classifying state $S_1^{(i)}$ at time $t$.
The passive cluster $S_2^{(i)}(t)$, $i=\overline{1,n}$, holds the objects of subsystem~2 whose last classifying state was $S_1^{(i)}$.
The reserve cluster $S_2^{(0)}(t)$ holds the objects of subsystem~2 that have held no classifying state since entering the system.
Every object belongs to exactly one of these $2n+1$ clusters at each instant.
\end{definition}

\begin{definition}[Multicluster open system]\label{def:mcsystem}
A multicluster open system is the collection of the $2n+1$ dynamic clusters of \cref{def:cluster} together with an external reservoir.
The reservoir supplies an exogenous inflow $\Delta N_2^{(0)}(t)$ over the interval $(t,t+1)$ and receives the objects that leave through a departure channel with destination index $j=n+1$.
\end{definition}

\subsection{Admissible movements}
Within one unit of time $\Delta t=1$ an object changes its cluster at most once.
An object in an active cluster $S_1^{(i)}(t)$ can move at time $t+1$ into a different active cluster $S_1^{(j)}(t+1)$ with $j\neq i$, remain in its current active cluster, pass to the passive cluster $S_2^{(i)}(t+1)$ of the same index, or leave the system through departure.
An object in a passive or reserve cluster $S_2^{(i)}(t)$, $i=\overline{0,n}$, can move at time $t+1$ into any active cluster $S_1^{(j)}(t+1)$, $j=\overline{1,n}$, remain in the same passive cluster $S_2^{(i)}(t+1)$, or leave through departure.
An object entering the system during $(t,t+1)$ appears at time $t+1$ in one of the active clusters $S_1^{(j)}(t+1)$, in the reserve cluster $S_2^{(0)}(t+1)$, or outside the system.
The retention shares $P^{**(i)}$ and $P^{*(i)}$ measure the mass that stays in place, and the reserve cluster $S_2^{(0)}$ keeps its mass through $P^{*(0)}$ like every other passive cluster.
\Cref{fig:scheme} shows these movements together with the departure and inflow channels.

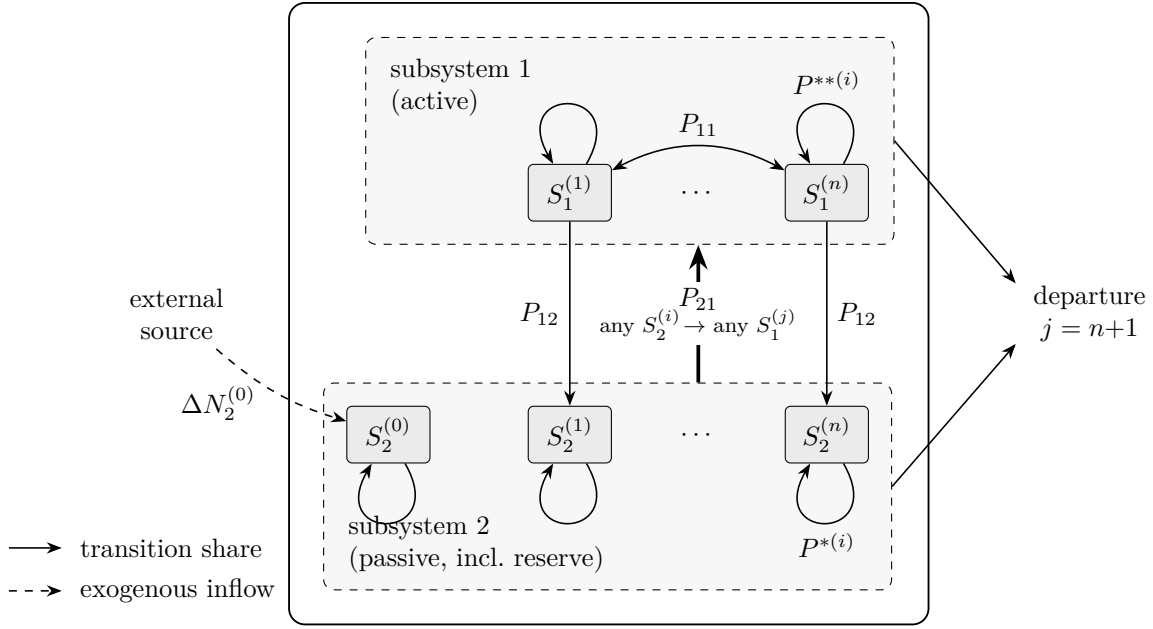
\begin{figure}[t]
\centering
\begin{tikzpicture}[>=Stealth,
  every node/.style={font=\small},
  clus/.style={draw, rounded corners=2pt, fill=black!8,
    minimum width=11mm, minimum height=7.5mm},
  arr/.style={->, semithick},
  block/.style={->, line width=1.4pt},
  darr/.style={->, semithick, dashed}]
  \node[clus] (a1) at (0,3.2)   {$S_1^{(1)}$};
  \node       (ad) at (1.7,3.2) {$\cdots$};
  \node[clus] (an) at (3.4,3.2) {$S_1^{(n)}$};
  \node[clus] (r0) at (-2.4,0)  {$S_2^{(0)}$};
  \node[clus] (p1) at (0,0)     {$S_2^{(1)}$};
  \node       (pd) at (1.7,0)   {$\cdots$};
  \node[clus] (pn) at (3.4,0)   {$S_2^{(n)}$};
  \draw[arr] (a1) to[out=60,in=120,looseness=7] (a1);
  \draw[arr] (an) to[out=60,in=120,looseness=7]
    node[above] (lpp) {$P^{**(i)}$} (an);
  \draw[arr] (r0) to[out=-60,in=-120,looseness=7] (r0);
  \draw[arr] (p1) to[out=-60,in=-120,looseness=7] (p1);
  \draw[arr] (pn) to[out=-60,in=-120,looseness=7]
    node[below] (lps) {$P^{*(i)}$} (pn);
  \coordinate (aw) at (r0 |- a1);
  \begin{scope}[on background layer]
    \node[draw, dashed, rounded corners=4pt, fill=black!3, inner sep=3mm,
      fit=(a1)(an)(lpp)(aw)] (abox) {};
    \node[draw, dashed, rounded corners=4pt, fill=black!3, inner sep=3mm,
      fit=(r0)(pn)(lps)] (pbox) {};
  \end{scope}
  \node[anchor=north west, align=left]
    at ([xshift=2mm,yshift=-1.5mm]abox.north west) {subsystem 1\\(active)};
  \node[anchor=south west, align=left]
    at ([xshift=2mm,yshift=1mm]pbox.south west) {subsystem 2\\(passive, incl.\ reserve)};
  \draw[<->, semithick] (a1) to[bend left=28] node[above]{$P_{11}$} (an);
  \draw[arr] (a1) -- node[left]{$P_{12}$} (p1);
  \draw[arr] (an) -- node[right]{$P_{12}$} (pn);
  \draw[block] (pbox.north -| pd) -- node[fill=white, inner sep=2pt,
    align=center, pos=0.5]
    {$P_{21}$\\[-2pt] {\scriptsize any $S_2^{(i)}\!\to$ any $S_1^{(j)}$}}
    (abox.south -| ad);
  \draw[arr] (abox.east) -- (5.9,2.0);
  \draw[arr] (pbox.east) -- (5.9,1.2);
  \node[align=center, anchor=west] at (6.0,1.6) {departure\\$j=n{+}1$};
  \node[align=center] (src) at (-5.2,1.6) {external\\source};
  \draw[darr] (src) to[bend right=12]
    node[pos=0.35, below left, inner sep=1pt]{$\Delta N_2^{(0)}$} (r0);
  \begin{scope}[on background layer]
    \node[draw, line width=0.7pt, rounded corners=6pt, inner sep=4.5mm,
      fit=(abox)(pbox)] (sys) {};
  \end{scope}
  \matrix at (-5.6,-1.8) [draw=none, column sep=3pt, row sep=1pt,
    column 1/.style={anchor=center}, column 2/.style={anchor=west}]{
    \draw[arr] (0,0) -- (0.7,0); & \node{transition share}; \\
    \draw[darr] (0,0) -- (0.7,0); & \node{exogenous inflow}; \\
  };
\end{tikzpicture}
\caption{Admissible movements in the multicluster open system, drawn for general $n$; the dots stand for the clusters that are not shown.
The dashed boxes group the active clusters $S_1^{(i)}$ into subsystem~1 and the passive clusters $S_2^{(i)}$ together with the reserve $S_2^{(0)}$ into subsystem~2, and the solid border marks the boundary of the open system.
An active cluster sends mass to any other active cluster ($P_{11}$), to the passive cluster of the same index ($P_{12}$), or to departure, and the share $P^{**(i)}$ stays in place.
The block arrow stands for the moves $P_{21}$ from any passive or reserve cluster to any active cluster; each passive or reserve cluster also sends mass to departure, the share $P^{*(i)}$ stays in place, and no movement from the reserve to a passive cluster is admissible.
Every cluster keeps a share in place, shown as a self-loop on each, and the reserve retains through $P^{*(0)}$.
The departure arrows leave the subsystem boxes because every cluster can send objects out of the system to the destination $j=n+1$, and the dashed arrow carries the exogenous inflow $\Delta N_2^{(0)}$ from the external source into the reserve, the only cluster it enters directly; the added mass reaches the active clusters within the step through the reserve outflow $P_{21}^{(0,i)}$.}
\label{fig:scheme}
\end{figure}

\subsection{Indicators and the balance dynamics}
Let $\Omega(t)$ be the set of all objects in the system at time $t$.
Membership of an object $O\in \Omega(t)$ in a cluster is read through the indicator function
\begin{equation}\label{eq:indicator}
I_{S_j^{(i)}(t)}(O)=
\begin{cases}
1, & O\in S_j^{(i)}(t),\\
0, & O\notin S_j^{(i)}(t).
\end{cases}
\end{equation}
The head count of a cluster is the number of objects it holds,
\begin{equation}\label{eq:count}
N_j^{(i)}(t)=\sum_{O\in \Omega(t)} I_{S_j^{(i)}(t)}(O),
\qquad j\in\{1,2\},\ i_1=\overline{1,n},\ i_2=\overline{0,n}.
\end{equation}
The total counts of the two subsystems are
\begin{equation}\label{eq:totals}
N_1(t)=\sum_{i=1}^{n} N_1^{(i)}(t),
\qquad
N_2(t)=\sum_{i=0}^{n} N_2^{(i)}(t),
\end{equation}
and the whole system holds $N_1(t)+N_2(t)$ objects at time $t$.

Each cluster count evolves by a balance of the objects that enter and leave it.
The general form of the model reads
\begin{equation}\label{eq:gen_1_math}
N_j^{(i)}(t+1)=N_j^{(i)}(t)+Q_{\mathrm{in},j}^{(i)}(t)-Q_{\mathrm{out},j}^{(i)}(t),
\qquad j\in\{1,2\},\ i_1=\overline{1,n},\ i_2=\overline{0,n},
\end{equation}
where $Q_{\mathrm{in},j}^{(i)}(t)$ and $Q_{\mathrm{out},j}^{(i)}(t)$ are the inflow to and the outflow from cluster $S_j^{(i)}(t)$ over the interval $(t,t+1)$.
These aggregate flows split into contributions weighted by the transition indicators.
With $P_{11}^{(i,j)}(t)$ the share of an active cluster $i$ that moves to a different active cluster $j$ or to departure $j=n+1$, $P_{12}^{(i)}(t)$ the share of active cluster $i$ that moves to the passive cluster of the same index, and $P_{21}^{(i,j)}(t)$ the share of a passive or reserve cluster $i$ that moves to an active cluster $j$ or to departure, the flows detail as
\begin{align}
Q_{\mathrm{in},1}^{(i)}(t) &=
\sum_{j=1}^{n} N_2^{(j)}(t)\,P_{21}^{(j,i)}(t)
+\bigl[\Delta N_2^{(0)}(t)+N_2^{(0)}(t)\bigr]P_{21}^{(0,i)}(t)
+\sum_{\substack{j=1\\ j\neq i}}^{n} N_1^{(j)}(t)\,P_{11}^{(j,i)}(t),
\label{eq:model_e1}\\
Q_{\mathrm{out},1}^{(i)}(t) &=
N_1^{(i)}(t)\bigl[P_{12}^{(i)}(t)+P_{11}^{(i,n+1)}(t)\bigr]
+\sum_{\substack{j=1\\ j\neq i}}^{n} N_1^{(i)}(t)\,P_{11}^{(i,j)}(t),
\label{eq:model_e2}\\
Q_{\mathrm{in},2}^{(i)}(t) &= N_1^{(i)}(t)\,P_{12}^{(i)}(t),
\label{eq:model_e3}\\
Q_{\mathrm{out},2}^{(i)}(t) &= N_2^{(i)}(t)\sum_{j=1}^{n+1} P_{21}^{(i,j)}(t),
\qquad i=\overline{1,n},
\label{eq:model_e4}\\
Q_{\mathrm{in},2}^{(0)}(t) &= \Delta N_2^{(0)}(t),
\qquad
Q_{\mathrm{out},2}^{(0)}(t) =
\bigl[\Delta N_2^{(0)}(t)+N_2^{(0)}(t)\bigr]\sum_{j=1}^{n+1} P_{21}^{(0,j)}(t).
\label{eq:model_e5}
\end{align}
A transition to the destination index $j=n+1$ is read as an object leaving the system, and both subsystems reach it.
The inflow $\Delta N_2^{(0)}(t)$ is the exogenous parameter of the model, the number of objects added over $(t,t+1)$.
An object added during the interval reaches at time $t+1$ one of the active clusters or the reserve cluster $S_2^{(0)}(t+1)$.
The reserve takes no inflow from another cluster, and \eqref{eq:model_e5} records this: its only inflow is the exogenous $\Delta N_2^{(0)}(t)$.

Summing \eqref{eq:gen_1_math} over all clusters and using \eqref{eq:model_e1}--\eqref{eq:model_e5} gives the conservation law
\begin{equation}\label{eq:balance}
N_1(t)+N_2(t)+\Delta N_2^{(0)}(t)-N^{(n+1)}(t)=N_1(t+1)+N_2(t+1),
\end{equation}
where $N^{(n+1)}(t)$ is the number of objects that leave the system over $(t,t+1)$,
\begin{equation}\label{eq:departures}
N^{(n+1)}(t)=
\sum_{i=1}^{n} N_1^{(i)}(t)\,P_{11}^{(i,n+1)}(t)
+\sum_{i=1}^{n} N_2^{(i)}(t)\,P_{21}^{(i,n+1)}(t)
+\bigl[\Delta N_2^{(0)}(t)+N_2^{(0)}(t)\bigr]P_{21}^{(0,n+1)}(t).
\end{equation}
The count at time $t+1$ equals the count at time $t$ raised by the net exchange $\Delta N_2^{(0)}(t)-N^{(n+1)}(t)$ with the reservoir.
The model conserves the number of objects once the inflow and the departures are accounted for.

\subsection{Transition indicators}
The unknown of the model is the vector $\Pt(t)$ that collects all transition indicators.

\begin{definition}[Transition indicators]\label{def:indicators}
The transition indicator $P_{j_1 j_2}^{(i_1,i_2)}(t)$ is the probability that an object in cluster $S_{j_1}^{(i_1)}(t)$ at time $t$ occupies cluster $S_{j_2}^{(i_2)}(t+1)$ at time $t+1$.
The vector $\Pt(t)$ has the following components.
The block $P_{11}^{(i,j)}$, $i=\overline{1,n}$, $j\in\{1,\dots,n,n+1\}\setminus\{i\}$, holds the moves from an active cluster to a different active cluster or to departure, $n^2$ components in all.
The block $P_{12}^{(i)}$, $i=\overline{1,n}$, holds the moves from an active cluster to the passive cluster of the same index, $n$ components.
The block $P_{21}^{(i,j)}$, $i=\overline{0,n}$, $j\in\{1,\dots,n,n+1\}$, holds the moves from a passive or reserve cluster to an active cluster or to departure, $(n+1)^2$ components.
The retention shares $P^{*(i)}$, $i=\overline{0,n}$, hold the mass of each passive or reserve cluster that stays in place, $n+1$ components.
The retention shares $P^{**(i)}$, $i=\overline{1,n}$, hold the mass of each active cluster that stays in place, $n$ components.
The canonical order of the blocks is $P_{11}\to P_{12}\to P_{21}\to P^{*}\to P^{**}$, giving $m=2n^2+5n+2$ components.
\end{definition}

The transition indicators are shares of a source cluster, so they are nonnegative,
\begin{equation}\label{eq:model_nonneg}
P_{11}^{(i,j)}(t)\ge 0,\quad P_{12}^{(i)}(t)\ge 0,\quad P_{21}^{(i,j)}(t)\ge 0,\quad P^{*(i)}(t)\ge 0,\quad P^{**(i)}(t)\ge 0,
\end{equation}
and the outgoing mass of each source, closed by its retention share, sums to one,
\begin{align}
P_{12}^{(i)}(t)+\sum_{\substack{j=1\\ j\neq i}}^{n+1} P_{11}^{(i,j)}(t)+P^{**(i)}(t) &= 1,
\qquad i=\overline{1,n},
\label{eq:model_norm1}\\
\sum_{j=1}^{n+1} P_{21}^{(i,j)}(t)+P^{*(i)}(t) &= 1,
\qquad i=\overline{0,n}.
\label{eq:model_norm2}
\end{align}
The retention shares $P^{*}$ and $P^{**}$ turn the outgoing budget of each of the $2n+1$ sources into an equality, so \eqref{eq:model_norm1}--\eqref{eq:model_norm2} contribute $2n+1$ normalisation constraints.
Together with the $2n+1$ balance equations \eqref{eq:gen_1_math}, these constraints are assembled into the linear system $\Amat(t)\,\Pt(t)=\ncal(t,t+1)$ of \cref{sec:vector}.
The coefficient matrix $\Amat(t)$ has full row rank, so the system is consistent and severely under-determined, as shown in \cref{sec:nonident}.
The data fix the aggregate flows and leave the internal routing free.

\section{Vector form and the segmentation problem}\label{sec:vector}

The balance dynamics of \cref{sec:model} are linear in the transition indicators.
Collecting the indicators of one step into a single vector turns the recovery into a linear
system, which is the form we solve and analyse for the rest of the paper.

Fix two adjacent snapshots at times $t$ and $t+1$.
The observation vector $\ncal(t,t+1)\in\real^{p}$ stacks the increments of the cluster
counts over the interval, corrected in the reservoir row for the exogenous inflow
$\Delta N_2^{(0)}(t)$, together with a block of ones for the normalisations.
The unknown is the indicator vector $\Pt(t)\in\real^{m}$, written in the canonical order
$P_{11}\to P_{12}\to P_{21}\to P^{*}\to P^{**}$.
The balance and normalisation laws of \cref{sec:model} then read
\begin{equation}\label{eq:inverse}
  \Amat(t)\,\Pt(t)=\ncal(t,t+1),\qquad \Pt(t)\ge 0,
\end{equation}
with $\Amat(t)\in\real^{p\times m}$, $p=4n+2$, $m=2n^{2}+5n+2$.
The first $2n+1$ rows form the balance block $B(t)$.
These are the balance equations written in increments of the counts over $(t,t+1)$, so only
the transition components $P_{11},P_{12},P_{21}$ enter them, and they contribute
$m'=2n^{2}+3n+1$ columns.
The retention shares cancel in the increments, hence the $P^{*}$ and $P^{**}$ columns of the
balance block are zero.
Departure carries no balance equation of its own, so each departure column has a single
nonzero entry in the balance block, and these entries are what give $B(t)$ full row rank
(\cref{sec:nonident}).
The remaining $2n+1$ rows are the normalisations of the outgoing probabilities of each
source, closed to unity by the variables $P^{*},P^{**}$, which enter only these rows and with
coefficient one.
The full block form of $\Amat(t)$ is recorded in \cref{sec:appB}.

The informative content sits in the $2n+1$ balance equations acting on the $m'$ transition
unknowns.
For $n=20$ this is $41$ equations in $861$ unknowns.
The system is severely under-determined, and \eqref{eq:inverse} has to be treated as a
regularised inverse problem.

\begin{definition}[segmentation problem]\label{def:segmentation}
Given two adjacent snapshots, the segmentation problem is to find a vector $\Pt(t)$
that satisfies the linear system \eqref{eq:inverse} together with the normalisation and
non-negativity constraints of \cref{sec:model}.
The name follows the labour-market model in which the aggregate counts are segmented into the
underlying inter-cluster transition indicators.
\end{definition}

A solution of the segmentation problem feeds three uses.
For analysis it yields the flows between clusters, both their volumes and their densities
relative to the source or the target cluster, which describe the internal movement that the
aggregate counts hide.
For forecasting it supplies a time series of operators whose components carry their own
trends; extrapolating those trends and applying the recovered operator returns forecasts of
the cluster counts one step ahead, an approach that uses more structure than extrapolating
the counts directly.
For control it exposes the operator to intervention through the control vector $\Rt$ of
\cref{sec:action}, which encodes a known external forcing and supports what-if analysis.

The segmentation problem is ill-posed, since \eqref{eq:inverse} admits a whole set of
non-negative solutions (\cref{sec:nonident}).
A single solution is selected by minimising the Tikhonov functional
\begin{equation}\label{eq:tikh}
  F_\alpha(x)=\norm{\Amat x-\ncal}^{2}+\alpha\,\norm{x-\bar x}^{2},\qquad x\ge 0,
\end{equation}
by projected steepest descent, where $x\equiv\Pt$ throughout.
The reference vector $\bar x$ carries the prior.
The choice $\bar x=0$ returns the minimum-norm solution.
The choice $\bar x=P_0$ returns the inertial reference, where $P_0$ is the full-retention
vector with $P^{*}=P^{**}=1$ and zeros elsewhere, the vector of the identity operator $I$
(the subscript $0$ labels the reference and is not a time index).
The subject domain enters \eqref{eq:inverse} only through the observations.
The cluster counts $N(t)$ and the inflow $\Delta N_2^{(0)}(t)$ set the coefficients of the
balance rows of $\Amat(t)$, and the increments $N(t+1)-N(t)$ set the right-hand side
$\ncal(t,t+1)$.
Everything else in \eqref{eq:inverse} is domain-independent, so the labour market and a beam
phase space are two instances of one model.

\section{Non-identifiability and the geometry of the stabiliser}\label{sec:nonident}

Set the constraint $P\ge0$ aside for a moment.
The solutions of \cref{eq:inverse} form the affine solution manifold
\begin{equation}\label{eq:manifold}
  M=\{x_0+v:\ v\in\ker A\},
\end{equation}
where $x_0$ is any particular solution and $v$ ranges over the null space $\ker A$.
Its dimension is $\dim\ker A=m-p=2n^2+n$ by \cref{lem:rank}.
The non-negative solutions are the intersection of $M$ with the cone $x\ge0$, a bounded polytope described below.
A move along $\ker A$ redraws the internal routing of the flows, the operator's micro-structure, and leaves the observed balance $Ax$ fixed.

In the formal statements we write $A\equiv A(t)$, $x\equiv P(t)$ and $b\equiv\ncal(t,t+1)$, and we suppress the argument $t$ wherever a single instant is fixed.
The matrix $A\in\real^{p\times m}$ collects the balance and normalisation equations, with $p=4n+2$ and $m=2n^2+5n+2$.
Throughout this section and the corresponding proofs we assume that all cluster populations at time $t$ are positive and that the exogenous inflow satisfies $\Delta N_2^{(0)}(t)\ge0$.
The mass that the reserve routes over the step is then the positive effective reserve mass $N^{(0)}_{\mathrm{eff}}(t)=\Delta N_2^{(0)}(t)+N_2^{(0)}(t)$ of \cref{eq:model_e5}, and wherever a source mass enters a formula below, the reserve enters through $N^{(0)}_{\mathrm{eff}}(t)$.

Three examples show what the freedom in $\ker A$ costs.
Take two clusters with the aggregate snapshot $(100,100)\to(100,100)$.
The family
$Q=\left[\begin{smallmatrix}1-s&s\\ s&1-s\end{smallmatrix}\right]$
reproduces this observation for every retention $1-s$ in $[0,1]$.
The minimum-norm solution picks $s=0.5$, a retention of $0.5$ and the fully mixed operator
$Q=\left[\begin{smallmatrix}0.5&0.5\\ 0.5&0.5\end{smallmatrix}\right]$,
while the inertial reference drives $s\to0$, a retention approaching $1$ and $Q\to I$.
This two-state $Q$ is the reduced picture of the retention freedom, which in the full model lives in the $P^{*},P^{**}$ shares and the $\ker A$ directions of \cref{lem:kernel}.
The same freedom is the classical non-identifiability of gross flows inside net flows \citep{BlanchardDiamond1990,Kalbfleisch1983} and of a trip matrix inside its margins \citep{VanZuylen1980}.
On the labour panel the recovered retention rises from $0.08$ to about $1.0$ at zero increment residual, so the point micro-structure is one arbitrary representative of the manifold (\cref{fig:examples}).

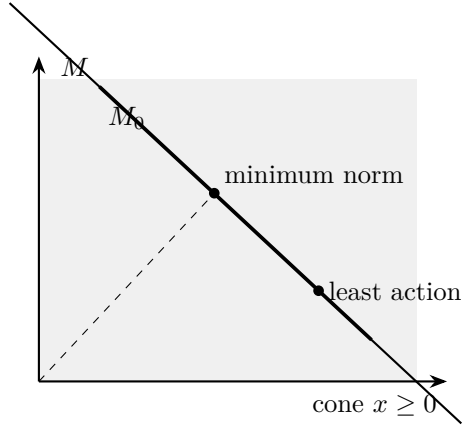
\begin{figure}[t]\centering
  \begin{tikzpicture}[>=Stealth, font=\small]
    \fill[black!6] (0,0) rectangle (5,4);
    \draw[->,thick] (0,0) -- (5.4,0);
    \draw[->,thick] (0,0) -- (0,4.3);
    \node[anchor=north east] at (5.4,-0.05) {cone $x\ge0$};
    \coordinate (P1) at (0.8,3.9);
    \coordinate (P2) at (4.4,0.55);
    \draw[thick] ($(P1)!-0.33!(P2)$) -- ($(P2)!-0.33!(P1)$);
    \node[anchor=south east] at (P1) {$M$};
    \draw[line width=1.4pt] (P1) -- (P2);
    \node[anchor=south east] at (1.55,3.2) {$M_0$};
    \coordinate (mn) at (2.32,2.49);
    \draw[dashed] (0,0) -- (mn);
    \fill (mn) circle (2pt);
    \node[anchor=south west] at (mn) {minimum norm};
    \coordinate (la) at (3.7,1.20);
    \fill (la) circle (2pt);
    \node[anchor=west] at (la) {least action};
  \end{tikzpicture}
  \caption{Non-identifiability in geometric form.
  The observation fixes the margins, so the consistent operators fill the segment $M_0$ where the manifold $M$ meets the non-negative cone.
  The minimum-norm reference selects the point nearest the origin, and least action selects the high-retention end.}
  \label{fig:examples}
\end{figure}

\begin{lemma}[full row rank]\label{lem:rank}
Let the populations of all clusters at time $t$ be positive.
Then $A$ has full row rank $p=4n+2$, the system $Ax=b$ is consistent, and its solution set
$M=\{x\in\real^m:\ Ax=b\}$ is an affine subspace of dimension $m-p=2n^2+n$.
\end{lemma}
The proof forms a square submatrix from the departure and closure columns and reads its rank off a block triangle with invertible diagonal blocks (\cref{sec:appA}).

The normalisation rows carry the structure behind the next bound.
Rows $i=2n+2,\dots,4n+2$ have entries in $\{0,1\}$ whose supports partition the whole column set.
Every column, including the closures $P^*,P^{**}$ and the columns of transitions from the reserve, enters exactly one normalisation.

\begin{lemma}[the cone]\label{lem:cone}
The null space meets the non-negative cone only at the origin, $\ker A\cap\real^m_+=\{0\}$.
More precisely,
\begin{equation}\label{eq:cone}
  \norm{Ax}_2\ \ge\ \norm{x}_2/\sqrt p\qquad(x\ge0).
\end{equation}
\end{lemma}
The proof opens the moduli of the normalisation rows without cancellation on the cone and compares the $\ell_1$ and $\ell_2$ norms (\cref{sec:appA}).
The constant $1/\sqrt p$ is a conservative bound on the cone, and the conditioning of the identifiable part is governed by the smallest non-zero singular value $\sigma_{\min}^+$ of $A$.

The admissible transitions form a bipartite graph on two disjoint vertex sets.
One side holds the $2n+1$ clusters as sources at time $t$, the other side holds the $2n+1$ clusters as targets at time $t+1$ together with departure, $4n+3$ vertices in all.
Each component of $P$ is one edge from a source to a target: $P_{11}$ joins the active sources to the other active targets and to departure, $P_{12}$ joins each active source to its passive target, $P_{21}$ joins the passive and reserve sources to the active targets and to departure, and the closures $P^*,P^{**}$ are the retention edges from the source copy of each cluster to its own target copy.
The graph is connected, since every source has an edge to the common sink departure.

\begin{lemma}[kernel structure]\label{lem:kernel}
For the fully admissible transition graph, $\ker A$ is spanned by balanced $4$-cycles.
The mass shift $z_{i\to j}+z_{k\to l}-z_{i\to l}-z_{k\to j}$, which raises the flows $i\to j$ and $k\to l$ and lowers $i\to l$ and $k\to j$ by a common amount, leaves every marginal fixed and so lies in $\ker A$, and a fundamental family of $m-p$ such cycles is a basis.
The same conclusion holds for any admissible subgraph in which every source has an edge to departure and every target has at least one incoming edge, with $m$ read as the number of admissible edges.
The identifiable quotient $\real^m/\ker A$ is canonically isomorphic to the row space of $A$, the space of marginals: the data fix the aggregate in- and out-flows of the clusters and leave the internal routing free.
\end{lemma}
The proof passes to the flow coordinates $z^{(i,j)}=N^{(i)}(t)\,P^{(i,j)}$, identifies $\ker A$ with the cycle space of the source-to-target bipartite graph, and reads a basis off the spanning tree centred at departure (\cref{sec:appA}).
\Cref{fig:cycle} shows one balanced $4$-cycle, the elementary move that shifts mass while every marginal stays fixed.

\begin{figure}[t]\centering
  \begin{tikzpicture}[>=Stealth, font=\small,
    nd/.style={draw, circle, inner sep=1.5pt, minimum size=6mm}]
    \node[nd] (i) at (0,1.5) {$i$};
    \node[nd] (k) at (0,0)   {$k$};
    \node[nd] (j) at (3.2,1.5) {$j$};
    \node[nd] (l) at (3.2,0)   {$l$};
    \draw[->,thick] (i) -- node[above]{$+$} (j);
    \draw[->,thick] (k) -- node[below]{$+$} (l);
    \draw[->,thick] (i) -- node[pos=0.82,above right]{$-$} (l);
    \draw[->,thick] (k) -- node[pos=0.82,below right]{$-$} (j);
    \node[anchor=east] at (-0.5,0.75) {sources};
    \node[anchor=west] at (3.7,0.75) {targets};
  \end{tikzpicture}
  \caption{A balanced $4$-cycle in $\ker A$.
  Raising the flows $i\to j$ and $k\to l$ and lowering $i\to l$ and $k\to j$ by a common amount leaves every marginal unchanged.}
  \label{fig:cycle}
\end{figure}
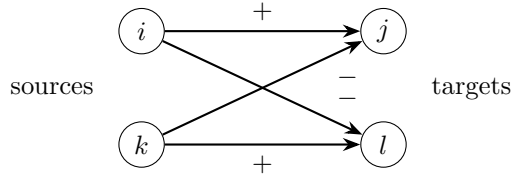

Non-identifiability of this kind is standard for Markov chains estimated from aggregate data \citep{Kalbfleisch1983,Kalbfleisch1984}.
Marginal data recover the equilibrium occupancies and the mean flows and leave the individual transition probabilities free.

\begin{corollary}[identifiable bridge coordinate]\label{cor:forest}
If an edge $e$ of the admissible transition graph lies on no balanced $4$-cycle, then the coordinate $x_e$ is constant on $M$ and fixed by the marginals, so $P$ is identifiable on $e$.
Under the hypotheses of \cref{lem:kernel} this happens exactly when $e$ is a bridge of the source-to-target bipartite graph.
In the full model the unique bridge is the retention edge of the reserve, and its identifiable coordinate has the closed form
\begin{equation}\label{eq:bridge-reserve}
  P^{*(0)}(t)=\frac{N_2^{(0)}(t+1)}{\Delta N_2^{(0)}(t)+N_2^{(0)}(t)}.
\end{equation}
\end{corollary}
The proof places the coordinate functional in the row space, expands any cycle through $e$ over the spanning $4$-cycles, and exhibits a balanced $4$-cycle through every edge of the full graph except one (\cref{sec:appA}).
The reserve is the one target with a single incoming edge, since no cluster sends mass into it, and its balance and normalisation give \cref{eq:bridge-reserve} directly.
Since the closure $P^{*(0)}(t)$ is a share in $[0,1]$, \cref{eq:bridge-reserve} is also a consistency condition on the data: the observations must satisfy $0\le N_2^{(0)}(t+1)\le\Delta N_2^{(0)}(t)+N_2^{(0)}(t)$, so over one step the reserve cannot grow beyond its effective mass.
A panel that violates this bound at some step admits no operator consistent with the balance model there.
One coordinate of the operator is therefore identifiable in closed form, while the remaining freedom has dimension $2n^2+n$.
On sparse variants with a restricted set of admissible transitions further bridges can appear, and each contributes an identifiable coordinate in the same way.

\begin{theorem}[non-identifiability]\label{thm:nonident}
The residual $F_0(x)=\norm{Ax-b}^2$ is continuous and convex on $\real^m$ and coercive on the cone $\real^m_+$.
Its minimiser set $M_0=\{x\in\real^m_+:\ F_0(x)=\min_{y\ge0}F_0(y)\}$ is non-empty, convex and compact.
In the feasible case $M_0=M\cap\real^m_+$ is a bounded polytope of dimension at most $2n^2+n$, with equality when $M\cap\real^m_{>0}\neq\varnothing$.
The set $M_0$ is either a single point or has the cardinality of the continuum, and it carries a unique element of minimal Euclidean norm, the normal solution.
\end{theorem}
The proof draws coercivity from \cref{eq:cone} and uniqueness of the normal solution from the strict convexity of the squared Euclidean norm; the dichotomy holds because a convex set with two distinct points contains the segment between them (\cref{sec:appA}).

For $2n^2+n>0$ the polytope $M_0$ is generically larger than a point, and the recovered operator $P$ is non-identifiable: a whole manifold of operators fits the observation equally well, of dimension $\dim M=820$ at $n=20$.
Full row rank and $\dim\ker A=2n^2+n$ are confirmed numerically for $n\in\{2,3,5,12,20\}$, with the smallest non-zero singular value $\sigma_{\min}^+$ ranging from $0.47$ at $n=20$ to $0.77$, which keeps the identifiable part stable.
For every $\alpha>0$ the regularised problem \cref{eq:tikh} has a unique solution by strict convexity \citep{Tikhonov1977}, and the next lemma identifies its limit.

\begin{lemma}[regularisation limit]\label{lem:limit}
Let $M_0$ be non-empty and feasible and let $x_\alpha$ minimise \cref{eq:tikh} over $x\ge0$.
As $\alpha\to0$ the family $x_\alpha$ converges to the unique point of $M_0$ closest to the reference $\bar x$ in the Euclidean norm.
\end{lemma}
The proof compares $x_\alpha$ with the projection of $\bar x$ onto $M_0$, which bounds the residual and the distance to $\bar x$ along the family, and identifies every accumulation point with that projection (\cref{sec:appA}).
For $\bar x=0$ the limit is the normal solution; other references reach the same point whenever their projection onto $M_0$ lands on it.

\begin{theorem}[reference as a Bregman projection]\label{thm:bregman}
Let $\bar x$ be a reference vector and let the recovered operator be $x=\arg\min_{Ax=b}D(x,\bar x)$ for a divergence $D$.
(a) For the Euclidean divergence $D(x,\bar x)=\norm{x-\bar x}^2$ this is the orthogonal projection of $\bar x$ onto the manifold $M$,
\begin{equation}\label{eq:proj}
  x=\bar x-A^{\!\top}(AA^{\!\top})^{-1}(A\bar x-b),
\end{equation}
and for $\bar x=0$ it is the affine minimum-norm point, which coincides with the normal solution of the polytope $M_0$ when that point is feasible.
(b) For the entropic Kullback-Leibler divergence assume the reference is componentwise positive, $\bar x>0$.
The recovered operator is then the $I$-projection of $\bar x$ onto $M\cap\real^m_+$ \citep{Csiszar1975}, the maximum-entropy coupling relative to $\bar x$: it exists because the divergence is finite and lower semicontinuous on the compact feasible set, and it is unique by strict convexity.
Under Csisz\'ar's condition $M\cap\relint\Delta\neq\varnothing$, where $\Delta$ is the corresponding probability polytope, the product of the source outflow simplices, the projection has the exponential form with finite multipliers and is computed by alternating Bregman projections onto the constraint blocks \citep{BregmanBCCNP2015}; for a uniform $\bar x$ it is the maximum-entropy operator.
Both are the Bregman projection of $\bar x$ onto $M$ for the chosen divergence \citep{Bregman1967}.
\end{theorem}
The proof of (a) is Lagrangian stationarity with the full-rank $AA^{\!\top}$ from \cref{lem:rank}; part (b) is the $I$-projection under linear constraints, realised by alternating Bregman projections over the constraint blocks (\cref{sec:appA}).
\Cref{fig:proj} shows the Euclidean case, the reference vector dropped onto the manifold along the perpendicular.

\begin{remark}
The implementation minimises \cref{eq:tikh} under $x\ge0$ by projected steepest descent.
The closed form \cref{eq:proj} is its exact solution when no bound is active, and otherwise \cref{eq:proj} is the interior iterate that the projection corrects.
The Kullback-Leibler variant applies to positive references and is computed by the alternating Bregman projections of \cref{thm:bregman}.
It excludes the inertial reference $P_0$: the balance rows annihilate the support of $P_0$, every feasible $x$ with a nonzero increment carries mass outside that support, and $D(x,P_0)=+\infty$ on all of $M$, so no $I$-projection exists and the inertial reference lives in the Euclidean geometry.
\end{remark}

\begin{figure}[t]\centering
  \begin{tikzpicture}[>=Stealth, font=\small]
    \draw[thick] (0,0.4) -- (5,2.1);
    \node[anchor=west] at (5,2.1) {$M$};
    \coordinate (xb) at (1.6,2.7);
    \fill (xb) circle (2pt);
    \node[anchor=south] at (xb) {$\bar x$};
    \coordinate (foot) at (2.14,1.13);
    \draw[dashed] (xb) -- (foot);
    \fill (foot) circle (2pt);
    \node[anchor=north west] at (foot) {recovered $x$};
    \coordinate (pa) at ($(foot)!2.5mm!(xb)$);
    \coordinate (pb) at ($(foot)!2.5mm!(5,2.1)$);
    \draw ($(pa)+(pb)-(foot)$) -- (pa) ($(pa)+(pb)-(foot)$) -- (pb);
  \end{tikzpicture}
  \caption{The reference vector as a projection.
  For the Euclidean divergence the recovered operator is the orthogonal projection of the reference $\bar x$ onto the solution manifold $M$.}
  \label{fig:proj}
\end{figure}
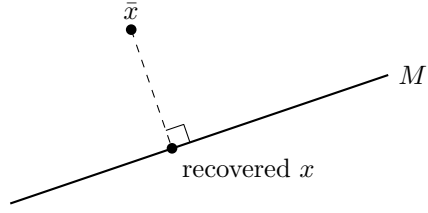

\Cref{lem:rank,lem:cone,lem:kernel,thm:nonident,thm:bregman} together fix the geometry.
The data determine the manifold $M$, and the pair (reference, divergence geometry) selects the representative on it.
On the outflow simplex of a single source the minimum-norm prior $\bar x=0$ gives the uniform outflow, the Shannon maximum-entropy prior there (\cref{fig:ris1}), and on inertial systems this underestimates retention.
Entropic optimal transport does not remove this choice.
It parametrises the choice through the entropic weight $\varepsilon$ for a fixed ground cost, here the unit cost of changing cluster, $C_{ij}=1-\delta_{ij}$: a small $\varepsilon$ gives high retention and inertia, a large $\varepsilon$ gives the maximum-entropy default \citep{Schiebinger2019,Cuturi2013,Jaynes1957}.
This motivates a substantive prior, the subject of the next section.

\begin{figure}[t]\centering
  \includegraphics[scale=0.846]{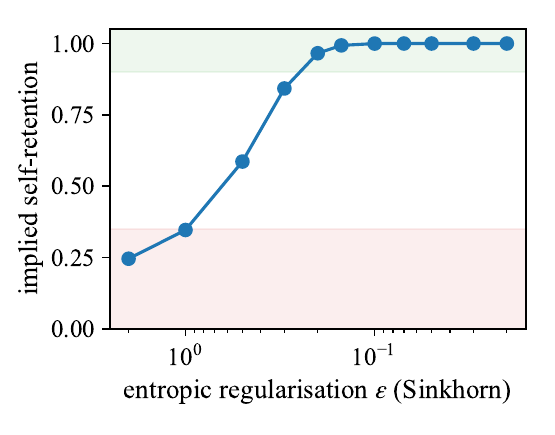}
  \caption{Optimal transport parametrises the prior.
  For the ground cost $C_{ij}=1-\delta_{ij}$ and matched marginals the recovered retention moves from the maximum-entropy default at large entropic weight to full retention at small entropic weight.}
  \label{fig:ris1}
\end{figure}

\section{The least-action prior on the operator trajectory}\label{sec:action}

The data fix the manifold of consistent operators, and a prior selects one point on it
(\cref{sec:nonident}).
This section states the prior that the paper argues for and situates it among the classical
principles it generalises.
The idea is a statement about the whole trajectory of operators: the rule of evolution
changes as little as the data allow.
We call it least action on the operator trajectory (informally, ``behavioural laziness'').
It selects the representative inside the manifold while leaving the observed component of the
dynamics untouched.

The prior rests on a distinction between two kinds of inertia (\cref{fig:staterule}).

\begin{definition}[inertia of the state]\label{def:inertia-state}
Inertia of the state is the requirement $N(t+1)\approx N(t)$.
It is the naive forecast, persistence, which sets the accuracy frontier on inertial series
\citep{MeeseRogoff1983,AtkesonOhanian2001,StockWatson2007}.
\end{definition}

\begin{definition}[inertia of the rule]\label{def:inertia-rule}
Inertia of the rule is the requirement $\Pt(t+1)\approx\Pt(t)$ for an operator $\Pt\neq I$.
It carries the signal about operator drift that the naive forecast cannot supply, since the
naive forecast constrains the state and says nothing about the rule.
\end{definition}

\begin{figure}[t]\centering
  \begin{tikzpicture}[>=Stealth, font=\small]
    \draw[->] (0,0) -- (4.6,0) node[right]{$t$};
    \draw[->] (0,-0.15) -- (0,1.4);
    \draw[thick] (0.15,0.85) -- (4.1,0.85);
    \node[left] at (0,0.85) {$N$};
    \node[anchor=south west] at (0.15,0.9) {inertia of the state};
    \begin{scope}[shift={(0,-2.3)}]
      \draw[->] (0,0) -- (4.6,0) node[right]{$t$};
      \draw[->] (0,-0.15) -- (0,1.7);
      \draw[thick] (0.15,0.55) -- (4.1,0.55);
      \node[left] at (0,0.55) {$P$};
      \draw[thick,dashed] (0.15,0.3) -- (4.1,1.3);
      \node[right] at (4.1,1.3) {$N$};
      \node[anchor=south west] at (0.15,1.2) {inertia of the rule};
    \end{scope}
  \end{tikzpicture}
  \caption{Two inertias.
  Inertia of the state holds $N(t)$ flat, which is the naive forecast.
  Inertia of the rule holds $P(t)$ flat while $N(t)$ drifts.}
  \label{fig:staterule}
\end{figure}
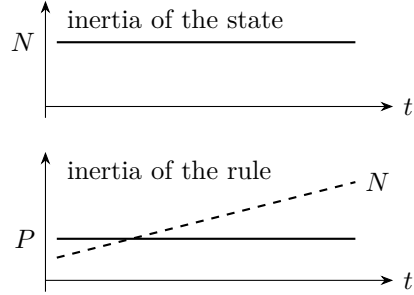

The running example of \cref{sec:nonident} shows the difference (\cref{fig:staterule}).
Inertia of the state holds the counts $(100,100)$ across the step.
Inertia of the rule holds the routing near its previous value, which keeps the retention $1-s$
close to where it was and away from the maximum-entropy value $0.5$ that the minimum-norm
solution assigns.
The two inertias act on different objects, and only the second reaches the operator.

The prior takes the form of a discrete action on the operator trajectory.

\begin{definition}[discrete action]\label{def:action}
The least-action prior selects the trajectory that solves
\begin{equation}\label{eq:action}
  \min_{\Amat(t)\Pt(t)=\ncal(t)}\ \sum_t\Bigl[\lambda_s\,D(\Pt(t),P_0)
  +\lambda_d\,D(\Pt(t),\Pt(t-1))-\langle \Rt, g(\Pt(t))\rangle\Bigr],
\end{equation}
where $\ncal(t)\equiv\ncal(t,t+1)$, $D$ is an operator-discrepancy measure, and
$\lambda_s,\lambda_d\ge 0$ weight the static and the dynamic term.
The static term $D(\cdot,P_0)$ prices any departure from the full-retention vector $P_0$, the
vector of the identity operator $I$, the cost of acting at all.
The dynamic term $D(\cdot,\Pt(t-1))$ prices a change of the rule between steps.
The map $g$ is a linear feature map onto the target components of the operator, and $\Rt$ is a
control vector that enters as a reward and encodes a known external forcing.
The implementation uses the Euclidean $D$, which admits an exact closed form and accepts the
boundary reference $P_0$; the entropic $D$ requires a positive reference
(\cref{thm:bregman}).
\end{definition}

For a linear $g$ the functional \eqref{eq:action} stays convex, and the control term sits in
the objective under the hard balance constraint, so it moves the solution only along
$\ket$ and never disturbs the observed balance.
All experiments below take $\Rt\equiv 0$ and leave $g$ unused; the control apparatus is
developed elsewhere.

\subsection{The connection to least action}\label{subsec:leastaction}

The form of \eqref{eq:action} is more than an analogy.

\begin{proposition}[discrete kinetic energy]\label{prop:bb}
In the Frobenius metric the discrete action
\[
  \sum_t\norm{\Pt(t)-\Pt(t-1)}^2
\]
is the time-discretisation of the flat analogue of the Benamou-Brenier kinetic energy of the operator path in the space of substochastic matrices.
\end{proposition}

\Cref{prop:bb} gives the condition of minimal operator change a precise variational meaning,
and its proof is in \cref{sec:appA}.
The classical fluctuation theory supplies a second reading, and it reaches both penalty
terms of the action.

\begin{remark}[Onsager-Machlup form of the action]\label{rem:onsager}
Take the Ornstein-Uhlenbeck reference diffusion
$dX=-\gamma(X-P_0)\,d\tau+\sqrt{2\varepsilon}\,dW$ on the operator space, drawn toward the
full-retention vector $P_0$.
Its Onsager-Machlup Lagrangian is $\tfrac{1}{4\varepsilon}\norm{\dot X+\gamma(X-P_0)}^2$
plus a drift-divergence correction that is constant for this linear drift
\citep{OnsagerMachlup1953}.
The cross term of the square is a total derivative and contributes only at the ends of the
horizon, so the interior part of the action is the Euclidean integral
\[
  \frac{1}{4\varepsilon}\int_0^T\bigl(\norm{\dot X}^2+\gamma^2\norm{X-P_0}^2\bigr)\,d\tau,
\]
kinetic energy plus potential energy.
Its time-discretisation with step $h$ consists of the two penalty terms of \eqref{eq:action}
with the Euclidean $D$, $\lambda_d=1/(4\varepsilon h)$ and $\lambda_s=\gamma^2h/(4\varepsilon)$,
and every pair of positive weights arises from a choice of $\varepsilon$, $\gamma$ and $h$
(\cref{sec:appA}).
The boundary term moves only the stationarity conditions at the ends of the horizon and
leaves the interior recursion of \cref{rem:el} unchanged, so for $\Rt\equiv0$ the prior
prices the operator trajectory by the discretised Onsager-Machlup action of an inertial
reference diffusion drawn toward $P_0$.
In discrete time the reading is exact: the two penalty terms are, up to an additive
constant, the negative logarithm of a Gaussian density on the trajectory, and minimisation
under the balance constraints picks the mode of the conditioned law.
The reference is Gaussian and ignores the constraint geometry, so $\Pt\ge0$ and the
admissible-transition structure stay exogenous restrictions; the discrete mode reading
survives them, since truncating a log-concave density to a convex set does not move the
constrained minimiser, while the continuum tube interpretation stops at an active bound.
\end{remark}

The functional \eqref{eq:action} states what the prior prices over the whole trajectory.

\begin{remark}[stationarity of the action]\label{rem:el}
For the Euclidean $D$, $\Rt\equiv0$ and no active bounds, stationarity of \eqref{eq:action}
in $\Pt(t)$ gives
\[
  \lambda_d\bigl(\Pt(t+1)-2\Pt(t)+\Pt(t-1)\bigr)
  =\lambda_s\bigl(\Pt(t)-P_0\bigr)+\Amat(t)^{\!\top}\mu(t),
\]
with $\mu(t)$ the multiplier of the balance constraint at time $t$.
This three-point recursion couples each operator to both neighbours, so the joint minimiser
of \eqref{eq:action} is a boundary-value object and depends on later observations.
The recursion is the discrete Euler-Lagrange equation of the Euclidean action of
\cref{rem:onsager}, whose kinetic-plus-potential form makes the boundary-value structure
expected.
\end{remark}

The recovered trajectory used for forecasting must stay causal.
The implementation therefore replaces the joint minimisation by a causal stepwise relaxation:
at each $t$ the past is frozen and the dynamic cost is minimised against fidelity to the
fresh operator on the fiber of the current observation, which yields the recursion of
\cref{subsec:nullspace} (\cref{lem:prox}).
That scheme uses past observations only.

\subsection{The connection to optimal transport}\label{subsec:ot}

Least action between marginals is dynamic optimal transport, and its static form is the
Kantorovich coupling of least cost.
The maximum-entropy coupling is the static end of the entropic family that regularises it
(\cref{sec:appC} recalls the two constructions).
The feasible set of the segmentation problem carries the transport structure exactly.

\begin{proposition}[transportation-polytope form]\label{prop:transport}
In the flow coordinates $z^{(i,j)}=N^{(i)}(t)\,\Pt^{(i,j)}(t)$ the feasible set of one layer is
affinely isomorphic to a transportation polytope with the observed marginals and a slack sink
for departure, with two qualifications: the admissible-transition graph imposes structural
zeros, since each passive target admits two entries and the reserve target one, and the
reserve row carries the effective mass $N^{(0)}_{\mathrm{eff}}(t)$.
The minimum-norm reference is the weighted-quadratic transport plan on this polytope with
source weights $1/(N^{(i)})^2$: it minimises $\sum_{i,j}(z^{(i,j)})^2/(N^{(i)})^2$, the
plan has the clipped potential form
$z^{(i,j)}=\max\bigl(a_i+(N^{(i)})^2 b_j,\,0\bigr)$, and it reduces to the standard
quadratically regularised ($\ell_2$) plan with additive potentials \citep{PeyreCuturi2019}
when the source masses are equal.
On the outflow simplex of a single source, the feasible set cut out by the normalisation
alone, the minimum-norm reference is the uniform outflow and coincides with the
maximum-entropy plan.
\end{proposition}

The proof is in \cref{sec:appA}.
For unequal source masses the weighted plan, the $\ell_2$ plan and the entropic (Sinkhorn)
plan \citep{Cuturi2013} select different points of the same polytope, and the coincidence on
the single-source simplex does not survive the coupling that the column constraints induce.
\Cref{prop:transport} realises the minimum-norm reference as a plan on the transportation
polytope; the inertial reference is a reference vector, and it selects its representative on
the same polytope through the Euclidean projection of \cref{thm:bregman}.
\Cref{fig:ris1} sweeps that polytope through the entropic weight between the maximum-entropy
default and full retention.

\subsection{The connection to the Schr\"odinger bridge}\label{subsec:bridge}

The soft version of rule inertia relates to the Schr\"odinger bridge, the most likely
evolution between two marginals under a reference dynamics (\cref{sec:appC}).
The rigorous statement is an algebraic identity.

\begin{lemma}[temporal smoothing as a proximal step]\label{lem:prox}
Temporal smoothing $\Pt_s(t)=(1-\lambda)\Pt(t)+\lambda\,\Pt_s(t-1)$ is the causal proximal
step
\[
  \Pt_s(t)=\arg\min_{x}\bigl[(1-\lambda)\norm{x-\Pt(t)}^2
  +\lambda\norm{x-\Pt_s(t-1)}^2\bigr],
\]
whose second term is the dynamic term of \eqref{eq:action} with the past frozen, and it is
one explicit Euler step of size $h=1-\lambda$ of the relaxation $\dot X=-(X-\Pt)$.
Null-space smoothing is the same proximal step under the constraint
$\Amat(t)x=\Amat(t)\Pt(t)$, the recursion of \cref{subsec:nullspace}.
\end{lemma}

The proof is in \cref{sec:appA}.

\begin{remark}\label{rem:bridge}
In the fluid formulation the Schr\"odinger bridge adds a Fisher-information penalty to the
Benamou-Brenier energy \citep{Leonard2014,ChenGeorgiouPavon2016,ChenGeorgiouPavon2021}, and
its strong-friction reduction is a first-order relaxation of the same form as the flow in
\cref{lem:prox}.
The correspondence stops at this level.
A term-by-term equality of temporal smoothing with a one-step bridge discretisation would
need the Kramers-Smoluchowski reduction, which sits beyond this paper, so the bridge reading
stays heuristic.
\end{remark}

The optimal-transport and the Schr\"odinger-bridge readings are two least-action problems on
paths of distributions, and their actions differ by the Fisher-information term that entropic
regularisation adds.

\subsection{The connection to the transfer operator}\label{subsec:fp}

The physical instantiation of \cref{sec:physics} rests on a coarse-graining that produces the
same balance structure as the socio-economic model.
Coarse-graining an ensemble replaces the transport of its density by the action of a
transfer operator on cell masses, a standard construction of transfer-operator theory
(\cref{sec:appC}).
In the physical instantiation the dynamics is the deterministic flow of the rapidity
equation, so the transfer operator over a fixed lag is the Perron-Frobenius operator of the
time-$\Delta t$ flow map.

\begin{lemma}[mass conservation]\label{lem:mass}
Let $T\colon X\to X$ be non-singular with respect to a reference measure $m$, and let the
modelled domain $D=\bigcup_{i=1}^{K}B_i\subseteq X$ be partitioned into cells with
$m(B_i)>0$.
The Perron-Frobenius transfer operator of $T$ conserves the mass of densities.
Ulam's matrix on the cells, with entries $Q_{ji}=m(B_i\cap T^{-1}B_j)/m(B_i)$ indexed by the
source cell $i$ down each column, satisfies $\sum_j Q_{ji}=m(B_i\cap T^{-1}D)/m(B_i)\le1$.
The column sum equals one for a cell whose mass stays in $D$, and the deficit
$1-\sum_j Q_{ji}$ is the mass that $T$ carries out of $D$, the leakage into departure.
\end{lemma}

\begin{theorem}[coarse-graining yields the same balance form]\label{thm:coarse}
Project the transfer operator $P_{\Delta t}$ of the ensemble over a fixed lag $\Delta t$ onto
the indicator functions of the cells of a modelled domain, Ulam's method.
The result is a substochastic matrix whose recovery from two coarse densities is a balance
system of the generic balance-plus-normalisation form $\Amat\Pt=\ncal$, under-determined in
the same way, and $\ker\Amat$ is again spanned by balanced $4$-cycles whenever the admissible
cell transitions satisfy the hypotheses of \cref{lem:kernel}.
The identification of the cells with the active and passive clusters of \cref{def:cluster} is
a modelling convention of the physical instantiation, introduced in \cref{sec:physics}.
\end{theorem}

The proofs are in \cref{sec:appA}.
\Cref{lem:mass,thm:coarse} are the bridge between the two instantiations, and they explain
why the same solver recovers the electron-ensemble operator without a single change
(\cref{sec:physics}).
\Cref{fig:coarse} shows the coarse-graining that carries a phase-space flow into a
substochastic matrix and the balance system it satisfies.

\begin{figure}[t]\centering
  \begin{tikzpicture}[>=Stealth, font=\small]
    \draw (0,0) rectangle (3,3);
    \foreach \y in {1,2} \draw[black!30] (0,\y) -- (3,\y);
    \draw[->,thick] (0.4,2.6) .. controls (1.6,2.3) and (2.1,1.2) .. (2.35,0.78);
    \fill (2.35,0.78) circle (2pt);
    \node[right] at (2.42,0.78) {attractor $\theta^*$};
    \node[below] at (1.5,-0.15) {phase space};
    \draw[->,very thick] (3.5,1.5) -- (4.7,1.5);
    \node[above] at (4.1,1.55) {Ulam};
    \begin{scope}[shift={(5.1,0.3)}]
      \fill[black!5] (0,0) rectangle (2.4,2.4);
      \draw (0,0) rectangle (2.4,2.4);
      \foreach \x in {0.8,1.6} \draw (\x,0) -- (\x,2.4);
      \foreach \y in {0.8,1.6} \draw (0,\y) -- (2.4,\y);
      \node[above] at (1.2,2.4) {$Q$ (substochastic)};
      \node[below] at (1.2,-0.15) {$\Amat\Pt=\ncal$};
    \end{scope}
  \end{tikzpicture}
  \caption{Coarse-graining an ensemble into a transfer operator on cell masses gives the same
  under-determined balance form as the socio-economic model.}
  \label{fig:coarse}
\end{figure}
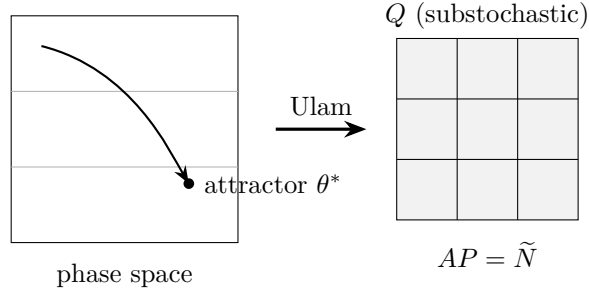

The control vector $\Rt$ completes the construction.
For a linear $g$ its term is convex and moves the solution only inside $\ket$, so it steers
the unobservable component of the operator while the observed balance stays exact.
This supports what-if analysis of a known forcing.
The experiments below keep $\Rt\equiv 0$.

\section{Numerical algorithms}\label{sec:num}

Both operationalisations of the prior are causal smoothers of the operator series, and both enter the tables below as regularisers.
Null-space smoothing keeps the observed balance exact by moving the operator only inside the null space of the constraints.
Temporal smoothing acts on the whole operator and gives up that guarantee.

\subsection{Null-space smoothing}\label{subsec:nullspace}

The correction of null-space smoothing lives in $\ket$.
The solver \eqref{eq:tikh} returns the unsmoothed operator $\Pt(t)$ for each window, minimum-norm by default with $\bar x=0$, and the smoother acts on top of this series.
It smooths the transition components through the balance block $B(t)$, and the closure variables $P^{*},P^{**}$ follow from the normalisations, so in full coordinates the shift lies in $\ker\Amat(t)$.
The causal projection of a direction $v$ onto $\ker B$ is
\begin{equation}\label{eq:proj-ker}
  \Pi_{\ker}(v)=v-B^{\!\top}(BB^{\!\top})^{+}Bv .
\end{equation}
The pseudoinverse form preserves the balance unconditionally,
\begin{equation*}
  B\,\Pi_{\ker}(v)=Bv-(BB^{\!\top})(BB^{\!\top})^{+}Bv=0 ,
\end{equation*}
because $Bv\in\range(BB^{\!\top})$ for every $v$.
In the transition coordinates $B$ is the $(2n+1)\times m'$ balance block on the transition columns, with $m'=2n^2+3n+1$, and $\dim\ker B=\dim\ker\Amat$ holds there, so the reduction to a square inverse is consistent.
The full row rank of $B$ comes from the invertible diagonal $V$-block in the departure-column argument of the proof of \cref{lem:rank}.
Under full row rank one has $(BB^{\!\top})^{+}=(BB^{\!\top})^{-1}$; a degenerate row, such as an industry with zero population, leaves the same guarantee in force through the pseudoinverse.
The smoothed series is the recursion
\begin{equation}\label{eq:recursion}
  \Pt_s(t)=\Pt(t)+\lambda\,\Pi_{\ker}\bigl(\Pt_s(t-1)-\Pt(t)\bigr) ,
\end{equation}
which preserves $\Amat(t)\Pt_s(t)=\Amat(t)\Pt(t)$ exactly, with a residual of about $10^{-15}$, and changes only the non-identifiable component.
The recursion step is the causal one-step projection
\begin{equation*}
  \Pt_s(t)=\arg\min_{x:\,\Amat(t)x=\Amat(t)\Pt(t)}\bigl[(1-\lambda)\norm{x-\Pt(t)}^2+\lambda\norm{x-\Pt_s(t-1)}^2\bigr] ,
\end{equation*}
the orthogonal projection of the convex combination $(1-\lambda)\Pt(t)+\lambda\Pt_s(t-1)$ onto the data fiber $M(t)$.
The argmin runs in the transition coordinates $P_{11},P_{12},P_{21}$, with the closures $P^{*},P^{**}$ fixed by the normalisations, so the norm above is the one on those coordinates.
This is the constrained proximal step of \cref{lem:prox}, taken forward in time, so the smoothed series uses past observations only.
The weight $\lambda$ carries the dynamic term of \eqref{eq:action}, the attraction to the past operator, and $1-\lambda$ weights fidelity to the fresh solution $\Pt(t)$.
The static term of \eqref{eq:action} does not enter this step.
It acts one level down, through the reference $\bar x$ of \eqref{eq:tikh} that produces $\Pt(t)$, and the default $\bar x=0$ switches it off.

\begin{lemma}[neutrality and leakage]\label{lem:neutral}
(i) For $\delta\in\ker\Amat(t)$ the macro forecast $\Amat(t)\Pt(t)$ is invariant under $\Pt(t)\mapsto\Pt(t)+\delta$.
(ii) The orthogonal projection $\Pi_{M(t)}(x)=x-\Amat^{\!\top}(\Amat\Amat^{\!\top})^{-1}(\Amat x-b)$ onto the fiber $M(t)$ preserves $\Amat(t)\Pi_{M(t)}(\cdot)=b$ exactly.
(iii) An increment $\delta\in\ker\Amat(t-1)$ applied at time $t$ violates the balance by $\Amat(t)\delta=\Delta A\,\delta$ with $\norm{\Amat(t)\delta}\le\norm{\Delta A}\,\norm{\delta}$, where $\Delta A=\Amat(t)-\Amat(t-1)$, so the leakage into the forecast is governed by the drift of the system matrix and vanishes when $\Amat$ is time-invariant.
\end{lemma}

The proof is in \cref{sec:appA}.

The map $\Pi_{\ker}$ is an orthogonal projector, idempotent and symmetric in the transition coordinates.
It annihilates the identifiable row-space part of the shift and touches only the component in $\ket$.
The row-space component of $\Pt_s(t)$ equals that of $\Pt(t)$ exactly, so the smoother leaves the identifiable part of the recovered operator untouched and acts only on the variance in $\ket$.
Leakage into a forecast one step ahead is governed by the drift of the system matrix (\cref{lem:neutral}(iii)) and vanishes when $\Amat$ is time-invariant.

Null-space smoothing keeps the correction inside $\ket$ and does not project onto the non-negative cone, so it can push individual components a little outside non-negativity.
On the real panels these violations are small.
About one percent of components are affected, the largest magnitude is $0.04$, and the balance is unchanged.
They sit in small cross-flows of the non-identifiable component in $\ket$.
Such violations cannot distort a data-determined conclusion.
The observed balance stays exact by \cref{lem:neutral}, and the affected directions are the ones the data leave free.
They share their source with the non-identifiable micro-structure, the freedom the solution manifold leaves open, and carry no separate meaning.
When strict non-negativity is required, an extra projection onto the bounded polytope $M\cap\real^m_+$ (\cref{thm:nonident}) removes them and preserves the marginals exactly, so the balance stays exact.
The reported runs omit this projection, because it breaks the closed-form causal recursion \eqref{eq:recursion} and the clean static and dynamic ablation, while the violations stay below the interpretive resolution of the non-identifiable component.

\subsection{Temporal smoothing}\label{subsec:temporal}

Temporal smoothing is causal exponential smoothing of the operator,
\begin{equation}\label{eq:ewma}
  \Pt_s(t)=(1-\lambda)\Pt(t)+\lambda\,\Pt_s(t-1) ,
\end{equation}
an analogue of \citet{BrandShare2017}.
It mixes the full operator and does not preserve the balance, so it serves as a baseline.
On the labour panel the balance residual is about $12\%$, rising to about $60\%$ in the physics instantiation of \cref{sec:physics}, against about $10^{-15}$ for null-space smoothing.

\subsection{Selecting the weight and cost}\label{subsec:lambda}

The weight $\lambda$ is chosen causally, by expanding-window cross-validation, a discrete analogue of the L-curve and generalised cross-validation \citep{Hansen1992,GolubHeathWahba1979}.
On the labour panel of eight yearly snapshots this choice returns $\lambda\to0$ almost everywhere, which switches the smoothing off, so these data carry no forecasting signal in the non-identifiable component.
For comparability the tables below also report the default $\lambda=0.5$, an interior point at which null-space smoothing keeps the balance exact and, on the synthetic map of \cref{sec:applic}, never raises the forecast error.

The projection is cheap.
Only $BB^{\!\top}$ of size $(2n+1)\times(2n+1)$ is inverted, which is $41\times41$ at $n=20$, and the $(2n^2+n)$-dimensional kernel is never materialised.
The block $B$ carries $O(1)$ nonzeros per column, so $BB^{\!\top}$ is formed in $O(n^2)$ from the sparse $B$ and inverted in $O(n^3)$; forming it from a dense $B$ would cost $O(n^4)$.
The per-step cost is $O(n^3)$ in the number of clusters.
As a diagnostic, the trajectory action $\sum_t\norm{\Pt(t)-\Pt(t-1)}^2$ measures the inertia of the rule, and both smoothers reduce it.

\section{Computational example: the Russian labour market}\label{sec:labour}

The first instantiation is the Russian labour market, and all verification here is causal.
The window expands over the yearly snapshots and each forecast uses only earlier data (rolling origin) \citep{Tashman2000}.
Four numbers summarise a forecast.
The mean absolute error (MAE) reports level accuracy.
The reliability $\eta_2$ is the share of the $2n+1$ indicators whose relative error stays at or below $2\%$.
It rewards a low dispersion of errors across the indicators, and it does not track the mean.
The correct sign share reports how often the forecast calls the direction of change.
The balance residual reports how far the recovered flows depart from the observed marginal constraint.
Two residuals are kept apart.
The solver residual $\norm{\Amat x-\ncal}/\norm{\ncal}$ measures fit accuracy.
The balance residual $\norm{\Amat\Pt_s-\Amat\Pt}/\norm{\Amat\Pt}$ measures how much a smoother disturbs the balance that the solver has already reached.
The regularisers compared are the raw minimum-norm solution, temporal smoothing (exponentially weighted moving average, EWMA), and null-space smoothing with weight $\lambda$; the operator components are extrapolated by trends under a common scheme (one trend for all components), with per-component and low-rank variants reported for robustness.

The model is the author's multicluster labour balance \citep{DrobotenkoNevecherya2021,DrobotenkoNevecherya2023}, built on the balance approach to labour-resource movement of \citet{Korovkin2001}.
The clusters are labour states.
$N_1$ is the count employed by industry, $N_2$ is the count unemployed by their last industry, and $N_2^{(0)}$ is the reserve of those entering and leaving the labour force.
Industries follow the national classifier OKVED.
The 2017 switch from OKVED-1 (12 industries) to OKVED-2 (20 industries) splits the series into a stable segment (Rosstat, 2005--2017) and a turbulent one (2017--2024) that contains the 2020--2022 shock.
Rosstat publishes employment by industry but not unemployment.
Aggregate unemployment is therefore split across industries in proportion to worker separations, by the method of \citet{UnempStruct2018}, and the resulting proxy agrees with the direct data where both exist (industry-share correlation $r=0.88$).
The error of this split enters as data uncertainty in the right-hand side $\ncal$.
The recovery preserves whatever balance the supplied $\ncal$ carries, so the proxy error propagates to the recovered operator as an ordinary data error and leaves the balance-exact property intact.
The separations used for the split are the same active-to-passive flows the operator recovers, so the imputed right-hand side is not independent of the recovered flows.
The identified aggregate flows therefore inherit the structure of the imputation model, and the $r=0.88$ check rests only on the subset of years with direct data.
Part of the statistical insignificance reported below may reflect this data noise together with the short series.

The abstract states of \cref{sec:model} take a concrete labour reading.
The classifying state $S_1^{(i)}$ is employment in industry $i$ under the OKVED-2 nomenclature, so $n=20$.
\Cref{tab:okved2} lists the twenty industries.
The professional, scientific and technical section is split into two, industry 13 excluding research and development and industry 14 for research and development, which fixes the count at twenty and drops the unstable extraterritorial category (Rosstat labour-force statistics).
The active cluster $S_1^{(i)}(t)$ is the set of workers employed in industry $i$ at time $t$.
The passive cluster $S_2^{(i)}(t)$ is the set of unemployed whose last employment was industry $i$.
The reserve cluster $S_2^{(0)}(t)$ holds those in the labour force with no recorded prior industry, the entrants and the never-classified.
The transition indicators read as labour flows \citep{DrobotenkoNevecherya2021,DrobotenkoNevecherya2023}.
$P_{11}$ is an industry-to-industry job change, a worker moving from employment in industry $i$ to employment in industry $j$.
$P_{12}$ is job loss, the move from the active cluster of industry $i$ to the passive cluster of the same index.
$P_{21}$ is hiring from unemployment or from the reserve into an industry.
The departure channel $j=n+1$ is leaving the labour force.
The retention shares carry the two persistences, $P^{*}$ the unemployment persistence of a passive cluster and $P^{**}$ the employment retention of an active cluster.

\begin{table}[t]
\refstepcounter{table}\label{tab:okved2}%
{\centering\footnotesize\bfseries Table~\thetable. The $n=20$ classifying states: employment by
OKVED-2 industry\par}
\smallskip
\centering\footnotesize\setlength{\tabcolsep}{4pt}
\begin{tabular}{|r|l|r|l|}
\hline
$i$ & Industry & $i$ & Industry \\ \hline
1 & Agriculture, forestry, hunting and fishing & 11 & Financial and insurance activities \\ \hline
2 & Mining and quarrying & 12 & Real estate activities \\ \hline
3 & Manufacturing & 13 & Professional, scientific and technical (excl.\ R\&D) \\ \hline
4 & Electricity, gas and steam supply & 14 & Scientific research and development \\ \hline
5 & Water supply, sewerage and waste & 15 & Administrative and support services \\ \hline
6 & Construction & 16 & Public administration and defence \\ \hline
7 & Wholesale and retail trade; vehicle repair & 17 & Education \\ \hline
8 & Transportation and storage & 18 & Human health and social work \\ \hline
9 & Accommodation and food service & 19 & Arts, entertainment and recreation \\ \hline
10 & Information and communication & 20 & Other service activities \\ \hline
\end{tabular}
\end{table}

The main verification period is the turbulent one.
Eight yearly OKVED-2 snapshots 2017--2024 give five holdouts 2020--2024 under the expanding window.
The headline result is deterministic and structural.
Null-space smoothing keeps the observed balance at machine zero, about $10^{-15}$, while temporal smoothing corrupts it by about $12\%$ (\cref{tab:baselines}, \cref{fig:ris2}b).
Null-space smoothing also keeps the reliability of the unsmoothed solution, $0.302$, while temporal smoothing lowers it to $0.268$ and the naive forecast sits at $0.273$ (\cref{fig:ris2}a).
On mean error the operator methods do not surpass the naive forecast, as expected for inertial stocks \citep{MeeseRogoff1983,AtkesonOhanian2001}.

\begin{figure}[t]\centering
  \includegraphics[scale=0.846]{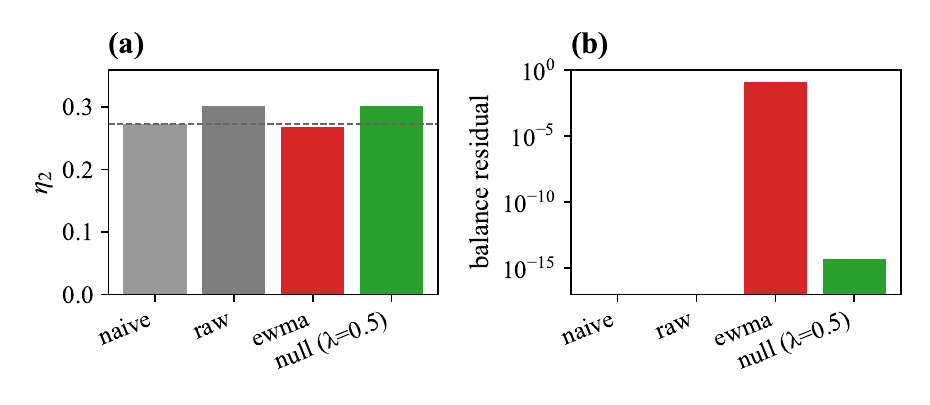}
  \caption{Labour market (OKVED-2). (a)~Reliability $\eta_2$, with the naive forecast as the
  dashed reference. (b)~Balance residual. Temporal smoothing corrupts the balance by about
  $12\%$; null-space smoothing holds it near $10^{-15}$.}
  \label{fig:ris2}
\end{figure}

A high retention default is also economically grounded.
The Russian labour market absorbs shocks mainly through wages and hours, a labour-hoarding response \citep{GimpelsonKapeliushnikov2011}, so industry self-retention is large and gross flows far exceed net ones \citep{DavisHaltiwangerSchuh1996,Shimer2005}.
The model is stock-flow consistent.
At every instant $\sum_i N_2^{(i)}+N_2^{(0)}=N_2$ and the operator conserves the head count, which is the economic reading of exact balance preservation.

A fair reading of point accuracy compares the prior against standard univariate baselines on the $2n+1$ indicator series, not against the naive forecast alone.
The baselines are a random walk with drift, a stationary AR(1), Holt smoothing (ETS), and a constant-share (shift-share) structural forecast; see \cref{tab:baselines}.
The baselines forecast the indicators directly, while the operator methods obtain the indicators from the recovered operator, an asymmetry that works against the prior.
The operator methods carry the safe default $\lambda=0.5$, since causal cross-validation on this panel switches the smoothing off and returns the raw solution.
The indicators span several orders of magnitude, so MAE is level-dominated; the scale-free sMAPE and RelMAE (error relative to the naive forecast, averaged across indicators) accompany it.
The RelMAE column scales each indicator error by the same-period naive error, so the naive forecast equals $1.00$ by construction; this same-period scaling can inflate when an indicator barely moves, so the ratios are read with that caveat.
This out-of-sample scaled error follows \citet{HyndmanKoehler2006}, and the sMAPE follows \citet{Makridakis1993}.
On both MAE and RelMAE the nominal minimum belongs to the random walk with drift ($46.4$, $\mathrm{RelMAE}=0.94$).
It is the only method below unity on RelMAE, and every operator method has $\mathrm{RelMAE}\ge1$.

\begin{table}[t]
\refstepcounter{table}\label{tab:baselines}%
{\centering\footnotesize\bfseries Table~\thetable. Baselines and operator methods (OKVED-2,
turbulent period, rolling origin, 5 holdouts)\par}
\smallskip
\centering\footnotesize\setlength{\tabcolsep}{4pt}
\begin{tabular}{|l|c|c|c|c|c|}
\hline
Method & MAE & sMAPE\% & RelMAE $\downarrow$ & $\eta_2$ & Sign \\ \hline
Naive (persistence)            & 48.7 & 11.6 & 1.00 & 0.273 & n/a \\ \hline
Random walk with drift         & 46.4 & 10.3 & 0.94 & 0.298 & 0.663 \\ \hline
AR(1) (stationary)             & 65.4 & 15.5 & 1.40 & 0.229 & 0.332 \\ \hline
ETS (Holt)                     & 53.5 & 12.8 & 1.12 & 0.283 & 0.580 \\ \hline
Shift-share (constant share)   & 50.3 & 11.0 & 1.02 & 0.259 & 0.537 \\ \hline
Raw (min-norm)                 & 51.8 & 12.2 & 1.08 & 0.302 & 0.595 \\ \hline
Ewma (temporal smoothing)      & 54.2 & 11.0 & 1.15 & 0.268 & 0.571 \\ \hline
Null ($\lambda{=}0.5$), proposed & 54.5 & 12.8 & 1.13 & 0.302 & 0.537 \\ \hline
\end{tabular}
\par\smallskip
{\raggedright\footnotesize Note. Only the random walk with drift falls below unity on RelMAE,
and its edge over the naive forecast is statistically insignificant. The naive forecast has
no defined change sign (persistence). Balance residual of the operator methods: raw $0$, ewma
$1.2\times10^{-1}$, null $5.2\times10^{-15}$; null-space smoothing preserves the balance to
machine zero.\par}
\par\smallskip
{\raggedright\footnotesize Diebold-Mariano note (exploratory). The Diebold-Mariano statistic
with the Harvey-Leybourne-Newbold small-sample correction \citep{DieboldMariano1995,%
HarveyLeybourneNewbold1997}, computed against the naive forecast on the holdout MAE over the
$n=5$ one-step-ahead evaluation points, gives: raw $p=0.542$, ewma $p=0.360$, null $p=0.447$,
random walk with drift $p=0.688$, AR(1) $p=0.192$, ETS $p=0.474$, shift-share $p=0.720$. These
values are exploratory. With only $n=5$ evaluation points the asymptotic null of the test is
unreliable even with the small-sample correction, so the valid verdict rests on the
within-holdout permutation test reported below. The comparison against the naive forecast is
nested, for which \citet{ClarkWest2007} is the standard alternative. \citet{Diebold2015}
cautions against Diebold-Mariano use in small samples, and \citet{GiacominiWhite2006} give a
conditional predictive-ability test suited to this rolling-window setting.\par}
\end{table}

The significance battery on the five holdouts does not establish an advantage in point accuracy.
A within-holdout sign-flip permutation test respects the cross-indicator dependence, a cluster bootstrap \citep{EfronTibshirani1993} assesses the reliability gain, and Holm-Bonferroni \citep{Holm1979} controls multiplicity.
No method shows a significant advantage over the naive forecast in point accuracy at any conventional level.
The within-holdout permutation test is two-sided, and the smallest $p$ of $0.19$ belongs to the stationary AR(1), whose large difference lies in the adverse direction and is worse than the naive forecast, so it counts as no advantage.
The random walk with drift is the only method that beats the naive forecast on point accuracy, and its edge is insignificant ($p=0.75$).
The cluster bootstrap of the reliability gain of null-space smoothing over the naive forecast returns a $95\%$ interval covering zero, so the test neither confirms nor refutes an effect, and the width of this interval is not read as evidence.
The directional component reaches significance only nominally.
The Pesaran-Timmermann test \citep{PesaranTimmermann1992} is one-sided for positive directional skill.
Pooling the $205$ directional calls, it flags the random walk with drift ($p<10^{-3}$).
Read on the two-sided absolute statistic, the same pool also flags the anti-skilled AR(1) as significant in the wrong direction ($p<10^{-3}$), which exposes the pooling as unreliable.
The pool treats the $41$ indicators at each of the five moments as independent.
Clustering by year, appropriate under the cross-indicator dependence and mindful of the small number of clusters \citep{CameronGelbachMiller2008}, removes the edge: the yearly correct-sign shares are heterogeneous, a year-cluster bootstrap interval covers $0.5$, and a year-block permutation is insignificant.
No method retains a significant directional edge.

The one deterministic advantage of the prior is the exact preservation of the observed balance, to about $10^{-15}$, while it recovers an interpretable transfer operator.
Univariate baselines cannot provide this in principle, since they do not recover the operator.
The value of the prior on this panel is therefore structural and theoretical.
On the stable OKVED-1 the structure yields no MAE gain over the naive forecast, temporal smoothing rescues the over-fitting per-component scheme, and the smaller indicator count (25 against 41) by itself eases reliability.
The point recovered operator is non-identifiable, and all macro conclusions are invariant to the choice of micro-structure.

\section{Computational example: a relativistic electron ensemble}\label{sec:physics}

The second instantiation is physical, and it carries the same model.
\Cref{thm:coarse} is the rigorous reason the same solver applies without modification.
The ensemble's density transports deterministically, and coarse-graining that transport into a Perron-Frobenius operator, by Ulam's method \citep{Ulam1960,DellnitzJunge1999,Froyland2014}, yields the same under-determined balance form $\Amat\Pt=\ncal$ as the stock-flow balance of the labour market.
Here the clusters are built from phase-space cells.
The rapidity axis is partitioned into $n=8$ cells, and these cells are the classifying states $S_1^{(i)}$ of \cref{def:states}.

We consider an ensemble of electrons in an intense laser field with radiation reaction of Landau-Lifshitz type.
The state is the rapidity $\theta=\operatorname{artanh}(v/c)$ \citep{AkintsovSymmetry2024,AkintsovPoP2025,AkintsovRPJ2024}, which evolves by the autonomous equation
\begin{equation}\label{eq:rapidity}
  \dot\theta=\alpha\tanh\theta-\varepsilon_{\mathrm{rad}}\sinh\theta\cosh\theta,
\end{equation}
where $\alpha\approx0.36$ is the secular phase-synchronised drive rate, which is intensity-independent, and $\varepsilon_{\mathrm{rad}}=(2r_e\omega/3c)\,a_0$ is the radiation-reaction parameter with $a_0=eE/(m\omega c)$ the normalised field amplitude.
The symbol $\alpha$ here is unrelated to the regularisation parameter.
The radiation-dominated regime is $a_0\gg1$.
The stable attractor is $\theta^*=\operatorname{arccosh}\sqrt{\alpha/\varepsilon_{\mathrm{rad}}}$; at $a_0=857$ this gives $\theta^*=5.93$ and $\gamma_L=\cosh\theta^*\approx188$ (about $96$~MeV).
The linear contraction rate toward the attractor has the closed form
\begin{equation}\label{eq:lambda}
  \Lambda=-\frac{\partial\dot\theta}{\partial\theta}\bigg|_{\theta^*}
  =\varepsilon_{\mathrm{rad}}\cosh 2\theta^*-\frac{\alpha}{\cosh^2\theta^*}
  =2(\alpha-\varepsilon_{\mathrm{rad}})\approx2\alpha=0.72,
\end{equation}
so it is set by the drive and is nearly intensity-independent.
The attractor at $\gamma_L\approx188$ sits far below the phase-locked drift ceiling $\gamma_{\max}\approx a_0=857$ for this configuration, which is the signature of radiation dominance.
At these intensities ($I\sim10^{24}$~W/cm$^2$) the quantum parameter $\chi\sim\gamma_L E/E_{\mathrm{cr}}$ reaches about $0.1$ in the phase-locked geometry, and up to about $0.4$ in the lab frame, where $E_{\mathrm{cr}}\approx1.3\times10^{18}$~V/m is the Schwinger critical field.
The classical Landau-Lifshitz reduction is the $\chi\to0$ limit.
At $\chi$ near $0.1$ the quantum reduction of the emitted power lowers the effective radiation-reaction parameter and shifts the mean attractor energy $\gamma_L$ by order twenty percent, while the contraction rate $\Lambda\approx2\alpha$ is set by the drive and stays quantum-insensitive.
The reproduced rate is therefore the quantum-insensitive observable, and $\gamma_L$ and $\theta^*$ are order-of-magnitude.
Quantum-stochastic broadening of the attractor is left unmodelled.

The flow is dissipative.
The coarse-grained mass is conserved while the phase-space volume is not, since the Perron-Frobenius transfer operator preserves the mass of densities by construction and encodes damping in the concentration of the density at $\theta^*$.
A finite partition resolves that concentration as a non-uniform distribution over the cells.
The eight cells span the rapidity axis from $\theta=0.8$ to the field ceiling $\theta_{\max}=\operatorname{arccosh} a_0\approx7.45$ in equal steps.
The state of an electron is its cell together with its drift status.
The active cluster $S_1^{(i)}(t)$ holds the electrons that occupy cell $i$ and still drift, those outside a settling band $|\theta-\theta^*|<0.45$ around the attractor.
The passive cluster $S_2^{(i)}(t)$ holds the electrons that have settled inside that band, indexed by the cell they last drifted in, which is the assignment of \cref{def:subsystems}.
Settled electrons concentrate at $\theta^*$, so the passive clusters carry most of their mass there.
This assignment of cells to the active and passive structure is a modelling convention of the physical instantiation, and \cref{thm:coarse} delivers the balance form without it.
$P^{**}$ is the active retention and $P^*$ the passive retention.
The reserve $S_2^{(0)}$ holds a sub-relativistic population below the lowest cell, and it enters the balance through the channels $P_{21}^{(0,\cdot)}$ of \cref{sec:model}.
The same solver recovers the coarse-grained transfer operator from two rapidity snapshots, the recovery having the same balance form (\cref{thm:coarse}).
The relaxation gives a non-stationary series of eight operators with solver residual $1.4\%$, and direct integration of the ensemble is the ground truth.
The recovered attractor cluster is populated and retains its mass ($P^*=0.924\pm0.026$ over the eight windows, and $0.91$ to $0.98$ under variation of the settling threshold and reservoir size), a qualitative signature that mass has concentrated at $\theta^*$.
The passive retention is tied to the non-identifiable outflows through the normalisation, so its value is read as a qualitative signature of the concentration at $\theta^*$.
The inertial reference $\bar x=P_0$ shifts the unobservable component of the operator to the absorbing attractor, raising the active-cluster retention $P^{**}$ from $0.114$ to $1.000$ (\cref{fig:ris3}a).
This illustrates the controllability of the non-identifiable component, since the reference $\bar x=P_0$ is chosen.
The balance-exact property is universal here too, with balance residual near $10^{-15}$ for null-space smoothing against $0.60$ for temporal smoothing (\cref{fig:ris3}b, \cref{tab:phys}).

The relaxation rate admits two estimates of a different kind.
From aggregates the gap is non-identifiable.
The gap converted to a rate $\Lambda_{\mathrm{rec}}=-\ln\rho^*/\Delta t$, with $\rho^*$ the second-largest eigenvalue modulus, is a property of the unobservable component.
Along the solution manifold (for the physics window $\dim\ker\Amat=137$, since a degenerate row on this window drops the rank of $\Amat$ to $33$ and lifts the defect one above $2n^2+n=136$) at a fixed fit residual (about $3.7\times10^{-3}$ for the window shown, distinct from the $1.4\%$ mean residual above) the gap sweeps the whole range from $0$, a representative with zero gap, to about $3.2\Lambda$.
The minimum-norm solution gives a representative with an artefact gap ($\rho^*\approx0.30$ at $\Delta t=0.6$, that is $\Lambda_{\mathrm{rec}}\approx2.8\Lambda$), and the scaling $\Lambda_{\mathrm{rec}}\propto1/\Delta t$ marks the gap as a property of the representative, since a physical rate would not depend on $\Delta t$.
This is a physical instance of the non-identifiability proved above: two adjacent snapshots do not identify the operator's spectral gap.
The micro-trajectory estimate returns the physical rate.
Given paired transitions of the same electrons between cells, which the aggregate setting does not provide, Ulam coarse-graining \citep{Ulam1960,DellnitzJunge1999,Froyland2014} recovers the gap.
In the resolved window (per-step displacement about two cell widths, $\Delta t=1.2$) it gives $\Lambda_{\mathrm{Ulam}}=0.68\pm0.02$, a ratio of $0.94$ to the analytic $\Lambda=0.72$, tending to $1$ as the ensemble concentrates at $\theta^*$.
The match tests the coarse-graining procedure on data from a known model, and it marks the information boundary of aggregate observation: what the micro-transitions carry, the aggregates do not recover.

\begin{table}[t]
\refstepcounter{table}\label{tab:phys}%
{\centering\footnotesize\bfseries Table~\thetable. Physics instantiation: recovery of the
electron-ensemble operator (radiation-dominated regime, $a_0=857$) and the bridge to the
dissipative attractor\par}
\smallskip
\centering\footnotesize\setlength{\tabcolsep}{4pt}
\begin{tabular}{|p{0.58\textwidth}|c|}
\hline
Quantity & Value \\ \hline
Solver residual $\norm{\Amat x-\ncal}/\norm{\ncal}$ (mean / max) & $1.4\%\,/\,2.4\%$ \\ \hline
Attractor retention $P^*$ (over 8 windows) & $0.924\pm0.026$ \\ \hline
Balance-exact residual, null ($\lambda{=}0.5$ / $\lambda{=}1$) & $2.7\times10^{-15}\,/\,5.7\times10^{-15}$ \\ \hline
Balance residual, ewma ($\lambda{=}0.5$) & $0.60$ \\ \hline
Bridge: $P^{**}$ at $\bar x=0\to\bar x=P_0$ ($\ker\Amat$ controllability) & $0.114\to1.000$ \\ \hline
Rate from micro-transitions (Ulam, $\Delta t{=}1.2$) / analytic $\Lambda{=}2(\alpha-\varepsilon_{\mathrm{rad}})$ & $0.68\pm0.02\,/\,0.72$ \\ \hline
Aggregate-recovery gap (along $\ker\Amat$) & non-ident.: $\Lambda_{\mathrm{rec}}\in[0,\,3.2\Lambda]$ \\ \hline
\end{tabular}
\end{table}

\begin{figure}[t]\centering
  \includegraphics[scale=0.846]{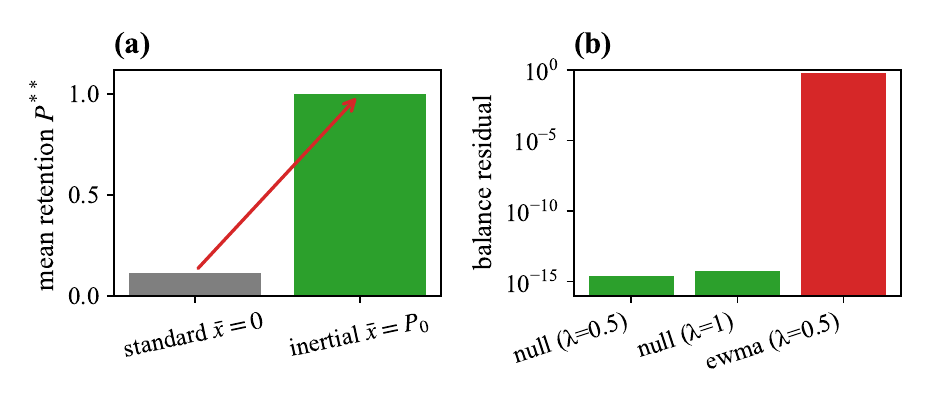}
  \caption{Physics instantiation. (a)~The inertial reference $\bar x=P_0$ raises $P^{**}$ from
  $0.114$ to $1.0$, the controllability of the unobservable component. (b)~Balance preservation
  to machine zero for null-space smoothing against temporal smoothing.}
  \label{fig:ris3}
\end{figure}

\section{Applicability map and robustness}\label{sec:applic}

The regime where the prior is useful is read from a controlled synthetic with the real balance structure and a known ground truth (\cref{fig:ris4}).
The axes are the operator trend $\delta$ and the recovery jitter $\rho$.

\begin{definition}[operator trend]\label{def:trend}
The operator trend $\delta$ is the magnitude of the systematic per-step change of the true rule.
In the controlled model the true operator moves along a fixed unit-norm direction $v$,
$\Pt^{\mathrm{true}}(t)=\Pt^{\mathrm{true}}(0)+t\,\delta\,v$, so
$\delta=\norm{\Pt^{\mathrm{true}}(t)-\Pt^{\mathrm{true}}(t-1)}$.
A value $\delta=0$ is a stationary rule.
A small $\delta$ is a behaviourally inertial system.
\end{definition}

\begin{definition}[recovery jitter]\label{def:jitter}
The recovery jitter $\rho$ is the relative amplitude of the per-step observation noise on the levels $N(t)$.
From the noisy levels a valid balance matrix $\Amat(t)$ is built, with the same structure as in the problem statement, and the operator is recovered from it.
The noise rotates $\ker\Amat(t)$, so the minimum-norm representative recovered from each snapshot random-walks in the unobservable subspace even when the true rule is unchanged ($\delta=0$).
In the full model the same rotation is carried by the natural evolution of $N(t)$ between steps ($\Delta A$, \cref{lem:neutral}).
This is the null-space noise that null-space smoothing removes.
\end{definition}

On the map null-space smoothing is independent of the trend, with error ratio at most one in every cell scanned.
By construction it touches only the unobservable component, a shift in $\ker\Amat$, and it leaves the data-consistent part untouched, the marginals and the clusters' in and out flows.
It can therefore suppress the jitter $\rho$ without biasing the systematic trend $\delta$.
Temporal smoothing averages the whole operator, including the observable part.
It yields a larger gain at weak drift and carries a harmful zone, up to $1.8$ times worse, at strong drift (large $\delta$) and weak jitter (small $\rho$), where lagging the trend outweighs the noise suppression.
The quality of the causal choice of $\lambda$ is assessed against an ideal.
The ideal $\lambda$ is chosen with hindsight on the control moments, minimising the forecast error on them, unavailable in practice and an upper bound on the achievable.
The causal choice nearly attains this bound, error ratio $1.06$ for null-space, and it beats the unsmoothed forecast.

\begin{figure}[t]\centering
  \includegraphics[scale=0.846]{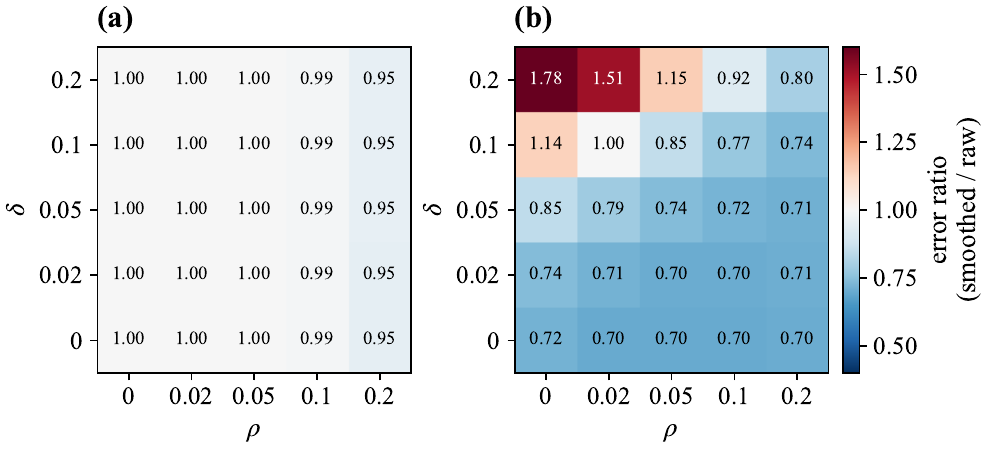}
  \caption{Applicability map (error ratio smoothed over raw; below one is beneficial; axes
  $\delta$, $\rho$). (a)~Null-space smoothing is safe everywhere. (b)~Temporal smoothing harms
  at strong $\delta$ and weak $\rho$.}
  \label{fig:ris4}
\end{figure}

Two causal null-space smoothers isolate the effect of an anchor and of temporal coordination (\cref{tab:ablation}).
The anchor pulls the operator toward the window-mean operator, and the coordination term pulls it toward the previous smoothed operator, which is the dynamic term of \eqref{eq:action}.
Both are neutral on reliability, $0.302$ for the anchor, as for raw, and $0.298$ for the coordination term.
On MAE the coordination term is preferable to the anchor, $54.41$ against $55.11$.
Both sit at a machine-zero balance residual.
Neither term yields a forecasting gain on this panel, in agreement with the outcome of the $\lambda$ cross-validation.

\begin{table}[t]
\refstepcounter{table}\label{tab:ablation}%
{\centering\footnotesize\bfseries Table~\thetable. Ablation of the anchor and coordination terms
(OKVED-2, common scheme, rolling-origin mean)\par}
\smallskip
\centering\footnotesize\setlength{\tabcolsep}{6pt}
\begin{tabular}{|l|c|c|c|}
\hline
Variant & MAE $\downarrow$ & $\eta_2\uparrow$ & Balance resid. \\ \hline
Raw                        & 51.76 & 0.302 & $0$ \\ \hline
Anchor (window mean)       & 55.11 & 0.302 & $1.3\times10^{-14}$ \\ \hline
Coordination (prev.\ op.)   & 54.41 & 0.298 & $2.0\times10^{-14}$ \\ \hline
\end{tabular}
\par\smallskip
{\raggedright\footnotesize Note. Both terms are neutral on reliability, and on MAE the
coordination term is preferable to the anchor. Each variant is a null-shift in $\ker\Amat$ at
$\lambda=1$, the anchor toward the window-mean operator, the coordination term toward the
previous $\Pt_s(t-1)$.\par}
\end{table}

The internal levers of recovery do not move the micro-structure.
The recovered retention is invariant to the reference vector, $\bar x\in\{0,P_0\}$, and to $\alpha\in[10^{-3},1]$, giving $0.083$ on OKVED-2 and $0.103$ on OKVED-1.
A substantive structure requires an external null-space move, as the non-identifiability analysis shows.
The reference vector is negligible here.
In the physical realisation $\bar x=P_0$ changes retention substantially.
The geometry of the data sets this difference; the formulation does not.
On the labour panel the data-consistent part dominates and the feasible polytope is wide, so the references $\bar x=0$ and $\bar x=P_0$ project to nearly the same operator; in the physics window the narrower feasible geometry lets the $P_0$ reference carry the unobservable component to the attractor.
The structural result is robust to the solver settings.
On the grid $\alpha\in\{10^{-3},10^{-2},10^{-1}\}$ crossed with $\tau\in\{10^{-3},10^{-4}\}$, with $\tau$ the solver convergence tolerance, the null-space reliability equals $0.302$ at every node but one, where it is $0.293$, at a balance residual near $10^{-14}$.

On the turbulent OKVED-2 a lower trajectory action goes with a lower MAE for null-space.
The correlation is $+0.55$ between the trajectory-action ratio (smoothed over raw) and MAE, pooled over extrapolation schemes, $\lambda$ and holdouts.
The link to reliability is weak on the extended panel, $+0.12$.
On the stable OKVED-1 there is no relationship, since there is nothing for the structure to catch.

\begin{remark}[James-Stein analogy]\label{rem:jamesstein}
The effect is statistical.
By analogy with the James-Stein effect \citep{JamesStein1961}, shrinkage toward the nearby operator $\Pt(t-1)$ reduces variance in the regime of small per-step change of the rule relative to the noise.
This serves as a guide only.
The James-Stein premises, unconditional Gaussian estimation and a fixed target, fail on the cone with a moving target, so no strict domination is claimed.
\end{remark}

A correct uncertainty estimate rests on the causal rolling-origin residuals, the out-of-sample forecast errors.
The spread of the extrapolation schemes does not calibrate, since their forecasts are too concordant and miss together, so coverage, the fraction of truths inside the interval, is only about $9\%$ at nominal $80$ to $90\%$.
The scheme-spread interval is the envelope of the point forecasts across the extrapolation schemes, so it carries no nominal-level parameter and its coverage stays at $0.091$ for every requested level.
The residual-based intervals are calibrated.
On synthetic data ($200$ realisations) they give per-indicator coverage of $0.834$ and $0.916$ at nominal $0.80$ and $0.90$ (\cref{tab:calib}).
Residual-based rolling-origin intervals are calibrated; the spread of extrapolation schemes is not.
All runs are deterministic given the seeds, and the solver settings and library versions are fixed in the published source code \citep{NevecheryaCode2026}.

\begin{table}[t]
\refstepcounter{table}\label{tab:calib}%
{\centering\footnotesize\bfseries Table~\thetable. Calibration of prediction intervals
(synthetic data, 200 realisations): residual-based intervals against the spread of schemes\par}
\smallskip
\centering\footnotesize\setlength{\tabcolsep}{6pt}
\begin{tabular}{|c|c|c|}
\hline
Nominal coverage & Residual-based & Scheme spread \\ \hline
0.80 & 0.834 & 0.091 \\ \hline
0.90 & 0.916 & 0.091 \\ \hline
\end{tabular}
\par\smallskip
{\raggedright\footnotesize Note. Per-indicator (marginal) coverage. The residual-based
rolling-origin intervals are calibrated near the nominal level, while the spread of the
extrapolation schemes covers far too little.\par}
\end{table}

\section{Discussion}\label{sec:disc}

The prior applies to systems with an inertial, slowly time-varying rule of evolution, a weak external forcing, and a closed mass balance.
Labour stocks, epidemics between waves, consumer preferences and dissipative plasma all sit in this class.
Mixing and chaotic systems sit outside it, and there the opposite maximum-entropy principle holds.
The remainder separates what the paper proves, the heuristic criterion for placing a system in the class, and the empirical limits of the labour study.

What the paper proves concerns the geometry alone.
The data determine only a manifold of operators of dimension $2n^2+n$ (\cref{lem:rank,lem:kernel}), so the spectral gap of a single recovered operator is not data-determined.
Along $\ker\Amat$ at an unchanged residual the gap sweeps the whole range $[0,\,3.2\Lambda]$ in the physics instantiation (\cref{tab:phys}).
The inverse problem is ill-posed.
Under a source condition $x^\dagger-\bar x=\Amat^{\!\top}w$ the standard estimate \citep{EnglHankeNeubauer1996} gives $\norm{x_\alpha-x^\dagger}=O(\sqrt\alpha)$, and the conditioning on the manifold is governed by the smallest non-zero singular value of $\Amat$.
Null-space smoothing keeps the balance neutral every period (\cref{lem:neutral}).

The heuristic criterion decides whether a given system belongs to the inertial class before any recovery.
The guide is the spectral gap $\gamma=1-\rho^{*}$, with $\rho^{*}$ the second-largest eigenvalue modulus, whose reciprocal $1/\gamma$ is the relaxation time in operator steps, up to a logarithmic factor \citep{LevinPeres2017}.
A relaxation time long against the forecast horizon marks a slow-mixing system that keeps the memory of its configuration, and there the inertial high-retention representative is appropriate.
A short relaxation time marks a fast-mixing system, and there maximum entropy is appropriate.
The apparent circle, that the gap is needed to choose the prior yet is itself computed from the recovered operator, does not arise.
The gap used for classification is read from the slow relaxation of the observed counts $N(t)$ toward equilibrium over a long trajectory, which a classical fact of spectral theory ties to the subdominant eigenvalue of the true operator \citep{LevinPeres2017}.
This coarse estimate is data-determined, available before recovery, and independent of the prior.
It is a different object from the per-window operator gap, which two adjacent snapshots do not identify.
In the short physics run the micro-trajectory estimate certified the rate, and the observed-series read was not used there.
An observed relaxation with no slow mode signals a non-inertial regime, and the prior does not apply.

The empirical limits are set by the labour study.
The holdout count is small, five for OKVED-2 at the limit of yearly data, so the significance tests are underpowered and do not establish an advantage over the naive forecast on any point-accuracy metric.
A systematic study of the reliability metric and its threshold is left to dedicated work, and stronger trend baselines such as damped-trend smoothing and the Theta method are left to future comparison.
The point micro-structure is non-identifiable and the temporal smoother is high-variance.
The rolling-residual intervals remain calibrated.
In the physics instantiation the recovered finite retention $P^{*}=0.924$ signals qualitatively that relaxation has occurred, the strongly inertial yet relaxing regime that the criterion selects.
An independent micro-trajectory estimate (Ulam's method) of the rate gives $\Lambda_{\mathrm{Ulam}}/\Lambda=0.94$.
The aggregate-recovery gap stays non-identifiable (\cref{tab:phys}).

The prior is distinct from two neighbouring approaches.
Entropic optimal-transport inference fixes the same marginals and works on the same feasible set \citep{Schiebinger2019,InverseOT2025}, and it differs in the direction of the prior, inertia with high retention against maximum entropy (\cref{fig:ris1}).
Soft-penalty temporal methods \citep{BrandShare2017} edit the whole operator, so any whole-operator smoother incurs a balance error, and on this panel temporal smoothing corrupts the observed balance by about $12\%$.
The present prior imposes the same rule inertia strictly inside $\ker\Amat$ and leaves the observed balance exact at about $10^{-15}$.
The confirmed value of the prior is deterministic and structural, exact balance preservation while it recovers an interpretable transfer operator.

\section{Conclusion}\label{sec:concl}

The main results of this work are as follows.
\begin{enumerate}
\item The geometry of the stabiliser in the inverse problem of transfer-operator recovery is
part of the a priori model of the system.
The data determine only a manifold of operators of dimension $2n^2+n$, while the pair of
reference and divergence geometry selects a representative on it.

\item We propose a least-action prior on the transfer-operator trajectory, causal and
preserving the observed balance exactly.
Its residual sits at machine zero, $10^{-15}$ to $10^{-14}$, where temporal smoothing, which
mixes the whole operator, incurs about $12\%$ on the labour panel and about $60\%$ in the
physical one.

\item We build an applicability map delimiting the class of systems with an inertial, weakly
forced rule of evolution.
On this map null-space smoothing is trend-independent and harms nowhere.

\item We demonstrate transferability across two qualitatively different instantiations, the
socio-economic and the physical.
The aggregate spectral gap is non-identifiable, sweeping $[0,\,3.2\Lambda]$ at an unchanged
residual, while an independent micro-trajectory estimate (Ulam's method) reproduces the
Landau-Lifshitz relaxation rate $\Lambda$ to within six percent.
On the labour market a rigorous significance analysis delimits the scope.
The confirmed value of the prior is structural and theoretical, and the point forecast shows
no significant gain.
\end{enumerate}

Several directions remain open.
A full scenario apparatus through the control vector $\Rt$ would turn the null-space freedom
into what-if analysis of a known forcing.
An external prior on the structure, from input-output tables, would constrain the
micro-structure that aggregate data leave non-identifiable.
Longer series, from quarterly or regional panels, would support a statistically significant
forecast evaluation and a causal choice of $\lambda$ on real data.

\appendix
\section{Proofs}\label{sec:appA}

\begin{proof}[Proof of \cref{lem:rank}]
Form a square submatrix $C\in\real^{p\times p}$ of $A$ from two groups of columns.
The first group holds the $2n+1$ departure columns, namely $P_{11}^{(i,n+1)}$ for $i=1,\dots,n$ and $P_{21}^{(i,n+1)}$ for $i=0,\dots,n$.
The second group holds the $2n+1$ closure columns $P^*$ and $P^{**}$.
Order the rows as the $2n+1$ balance rows followed by the $2n+1$ normalisation rows.
Then $C$ has the block form
\[
  C=\begin{pmatrix}V&\Theta\\ W&I\end{pmatrix},
\]
where $V$ is diagonal and carries the negative source masses, the entries $-N_1^{(i)}(t)$ for the active departure columns, $-N_2^{(i)}(t)$ for the passive ones and $-N^{(0)}_{\mathrm{eff}}(t)=-\bigl(\Delta N_2^{(0)}(t)+N_2^{(0)}(t)\bigr)$ for the reserve, because each departure column carries a single nonzero in the balance block; the block $\Theta$ vanishes, because the closures do not appear in the balance rows; and $I=I_{2n+1}$, because each closure sits in its own normalisation with coefficient $1$.
The populations are positive, so $V$ is invertible.
The matrix $C$ is block triangular with invertible diagonal blocks, so $\det C=\det V\cdot\det I_{2n+1}\neq0$ and $\rank C=p$.
$C$ is a $p\times p$ submatrix of the $p$-row matrix $A$, so $p\ge\rank A\ge\rank C=p$ and $\rank A=p$.
By the Kronecker-Capelli theorem $\rank(A\mid b)=\rank A$, the system $Ax=b$ is consistent, and its solution set is an affine subspace of dimension $m-\rank A=m-p=2n^2+n$.
\end{proof}

\begin{proof}[Proof of \cref{lem:cone}]
The normalisation rows $i=2n+2,\dots,4n+2$ have entries in $\{0,1\}$, and their supports partition the whole column set $\{1,\dots,m\}$.
Every column, including the closures $P^*,P^{**}$ and the reserve column, enters exactly one normalisation.
Take any $x\ge0$.
The modulus of each normalisation row opens without cancellation, since every entry of $x$ and every coefficient is nonnegative, so
\[
  \norm{Ax}_1\ \ge\ \sum_{i=2n+2}^{4n+2}\Bigl|\sum_j a_{ij}x_j\Bigr|
  \ =\ \sum_{i=2n+2}^{4n+2}\sum_j a_{ij}x_j
  \ =\ \sum_{j=1}^m x_j\ =\ \norm{x}_1,
\]
where the last equality uses that the normalisation supports partition the columns, so each $x_j$ is counted exactly once.
The norm comparisons $\norm{y}_2\ge\norm{y}_1/\sqrt p$ for $y\in\real^p$ and $\norm{x}_1\ge\norm{x}_2$ for $x\in\real^m$ give, for $x\ge0$,
\[
  \norm{Ax}_2\ \ge\ \norm{Ax}_1/\sqrt p\ \ge\ \norm{x}_1/\sqrt p\ \ge\ \norm{x}_2/\sqrt p,
\]
which is \cref{eq:cone} with constant $c=1/\sqrt p$.
In particular $Ax=0$ with $x\ge0$ forces $x=0$, so $\ker A\cap\real^m_+=\{0\}$.
\end{proof}

\begin{proof}[Proof of \cref{lem:kernel}]
We prove the two inclusions and then extract a basis.

A balanced $4$-cycle changes neither the row sums nor the column sums of the coupling, so it satisfies every balance and every normalisation equation and lies in $\ker A$.
This gives the inclusion $\supseteq$.

For the reverse inclusion pass to the flow variables $z^{(i,j)}=N^{(i)}(t)\,P^{(i,j)}$.
The populations are positive, so this is a diagonal change of coordinates and leaves the kernel dimension unchanged.
In these variables the normalisation of source $i$ is the row sum $\sum_j z^{(i,j)}=N^{(i)}(t)$, with the reserve entering through $N^{(0)}_{\mathrm{eff}}(t)$.
Adding the balance equation of target $j$ to the row sum of its own source $j$ turns the pair into the column sum $\sum_i z^{(i,j)}=N^{(j)}(t+1)$, the observed mass of target $j$ at the next instant.
Together these are the vertex-incidence relations of the source-to-target bipartite graph of admissible transitions, in which departure $j=n+1$ is a common sink with no column equation.
The missing equation at departure holds automatically on the kernel: the total of the row sums equals the total of the column sums, each edge being counted once at its source and once at its target, so a flow with zero sums at every source and at every target except departure also sums to zero at departure.
The kernel is therefore the cycle space of the whole graph, the flows with zero sums at every vertex, of dimension
\[
  |\text{edges}|-|\text{vertices}|+|\text{components}|=m-(4n+3)+1=m-p,
\]
consistent with \cref{lem:rank}.

Take the spanning tree centred at departure.
It contains every source-to-departure edge, present for each source and linking all sources to the sink, together with one edge $s_j\to j$ into each remaining target $j$.
The fundamental cycle of a non-tree edge $s\to j$ is the balanced $4$-cycle
\[
  z_{s\to j}-z_{s_j\to j}+z_{s_j\to\mathrm{dep}}-z_{s\to\mathrm{dep}}.
\]
Every fundamental cycle of this tree is a $4$-cycle because every source is adjacent to departure and $s\neq s_j$ for a non-tree edge.
There are $m-p$ non-tree edges, their fundamental cycles are linearly independent, and they span the cycle space, so balanced $4$-cycles span $\ker A$.
The argument uses two properties of the graph only: every source has an edge to departure, and every target has at least one incoming edge.
Any admissible subgraph with these properties is connected, admits the same star-shaped spanning tree, and has all its fundamental cycles balanced $4$-cycles through departure, which proves the extension claimed in the lemma.
The identifiable functionals are the ones constant on the cosets of $\ker A$, the elements of the annihilator $(\ker A)^{\perp}$, and this annihilator is the row space of $A$, the space of marginals.
\end{proof}

\begin{proof}[Proof of \cref{cor:forest}]
Suppose the edge $e$ lies on no balanced $4$-cycle.
The coordinate functional $x\mapsto x_e$ then vanishes on the spanning family of balanced $4$-cycles of \cref{lem:kernel}, hence on $\ker A$, so it lies in the annihilator $(\ker A)^{\perp}$, the row space of $A$.
A functional in the row space is constant on the affine manifold $M$, and its value there is a combination of the marginals, so $x_e$ is identifiable.

Next we show that under the hypotheses of \cref{lem:kernel} the edges on no balanced $4$-cycle are exactly the bridges of the bipartite graph.
Let $e$ lie on some cycle $C$.
The alternating unit flow around $C$ has zero sums at every vertex, so it lies in $\ker A$ and expands over the balanced $4$-cycles that span it.
Its component on $e$ is $\pm1$, so at least one $4$-cycle in the expansion has a nonzero component on $e$, and $e$ lies on that balanced $4$-cycle.
A bridge lies on no cycle, and a balanced $4$-cycle is a cycle, so the two exclusions coincide.

In the full model the reserve target has a single incoming edge, the retention edge $P^{*(0)}$: the blocks $P_{11}$ and $P_{21}$ end in active targets or departure, and $P_{12}$ ends in the passive targets of index $i\ge1$, so no transition edge enters the reserve.
An edge into a vertex of degree one lies on no cycle, so $P^{*(0)}$ is a bridge.
Every other edge lies on an explicit balanced $4$-cycle.
For an edge $s\to j$ with $j$ neither departure nor the reserve target, the target $j$ has a second incoming edge $s'\to j$ with $s'\neq s$: an active target receives from the $n-1$ other active sources, from all $n+1$ passive sources and from its own retention, and a passive target of index $i\ge1$ receives from the active source $i$ and from the passive source $i$.
Both sources have an edge to departure, so $z_{s\to j}-z_{s'\to j}+z_{s'\to\mathrm{dep}}-z_{s\to\mathrm{dep}}$ is a balanced $4$-cycle through $e$.
For an edge $s\to\mathrm{dep}$, pick a target $j$ admissible from $s$ other than departure and the reserve: an active or passive source retains into its own target copy, and the reserve source feeds any active target.
That target has a second incoming edge $s'\to j$ as above, and the four edges $s\to j$, $s'\to j$, $s'\to\mathrm{dep}$, $s\to\mathrm{dep}$ close a balanced $4$-cycle through $e$.
Hence $P^{*(0)}$ is the unique bridge for every $n\ge1$.

The closed form follows from the reserve balance.
\Cref{eq:gen_1_math} with \cref{eq:model_e5} gives
\[
  N_2^{(0)}(t+1)=N_2^{(0)}(t)+\Delta N_2^{(0)}(t)-N^{(0)}_{\mathrm{eff}}(t)\sum_{j=1}^{n+1}P_{21}^{(0,j)}(t),
  \qquad N^{(0)}_{\mathrm{eff}}(t)=\Delta N_2^{(0)}(t)+N_2^{(0)}(t).
\]
The reserve normalisation \cref{eq:model_norm2} turns the sum into $1-P^{*(0)}(t)$, so $N_2^{(0)}(t+1)=N^{(0)}_{\mathrm{eff}}(t)\,P^{*(0)}(t)$, which is \cref{eq:bridge-reserve}; the effective mass is positive by the standing assumption of \cref{sec:nonident}.
\end{proof}

\begin{proof}[Proof of \cref{thm:nonident}]
$F_0$ is the Euclidean norm composed with the affine map $x\mapsto Ax-b$, so it is continuous and convex on $\real^m$.
For coercivity on the cone take $x\ge0$ with $\norm{x}\ge\sqrt p\,\norm{b}$.
Then $\norm{Ax}\ge\norm{x}/\sqrt p\ge\norm{b}$ by \cref{eq:cone}, the triangle inequality gives $\norm{Ax-b}\ge\norm{Ax}-\norm{b}\ge\norm{x}/\sqrt p-\norm{b}\ge0$, and squaring the nonnegative sides gives
\[
  F_0(x)=\norm{Ax-b}^2\ \ge\ \bigl(\norm{x}/\sqrt p-\norm{b}\bigr)^2\ \longrightarrow\ \infty
\]
as $\norm{x}\to\infty$ with $x\ge0$.
The cone $\real^m_+$ is closed, and a continuous functional coercive on a closed set attains its minimum, so $M_0\neq\varnothing$.
$M_0$ is the intersection of a sublevel set of the convex $F_0$ with the convex cone $\real^m_+$, so it is convex; it is closed and bounded, so it is compact.

In the feasible case the minimum of $F_0$ is zero and $M_0=M\cap\real^m_+$.
Its recession cone is $\ker A\cap\real^m_+=\{0\}$ by \cref{eq:cone}, so $M_0$ is a bounded polytope.
Its dimension is at most $\dim M=2n^2+n$.
When $M$ meets the open orthant $\real^m_{>0}$ at a point $x_+$, a relative ball of $M$ around $x_+$ of small enough radius stays in the open orthant, so $M_0$ contains a set of dimension $\dim M$ and the two dimensions are equal.

For the cardinality, suppose $M_0$ contains two distinct points $x_0\neq x_1$.
By convexity it contains the whole segment $\{(1-\theta)x_0+\theta x_1:\theta\in[0,1]\}$, which has the cardinality of the continuum.
A one-point set has cardinality one, so $M_0$ is either a single point or has the cardinality of the continuum.

For the normal solution, the Euclidean norm attains a minimum value $r_0$ on the compact set $M_0$.
The set $M_S=\{x\in M_0:\norm{x}=r_0\}$ is non-empty and convex.
Suppose $x_0\neq x_1$ both lie in $M_S$.
The midpoint $x_{1/2}=\tfrac12(x_0+x_1)$ lies in $M_0$ by convexity, and strict convexity of the squared Euclidean norm \citep{Tikhonov1977} gives $\norm{x_{1/2}}^2<r_0^2$, hence $\norm{x_{1/2}}<r_0$, which contradicts the definition of $r_0$.
Hence $M_S$ is a single point, the unique normal solution.
\end{proof}

\begin{proof}[Proof of \cref{lem:limit}]
The set $M_0$ is non-empty, closed and convex, so the Euclidean projection $x^*$ of $\bar x$ onto $M_0$ exists and is unique by strict convexity.
The point $x^*$ is feasible for \cref{eq:tikh} and satisfies $Ax^*=b$, hence $F_\alpha(x^*)=\alpha\norm{x^*-\bar x}^2$.
Comparing $F_\alpha(x_\alpha)\le F_\alpha(x^*)$ and using that both terms of $F_\alpha$ are non-negative gives the two estimates $\norm{Ax_\alpha-b}^2\le\alpha\norm{x^*-\bar x}^2$ and $\norm{x_\alpha-\bar x}\le\norm{x^*-\bar x}$.
The second estimate confines the family to a fixed ball, so it is bounded.
Take any sequence $\alpha_k\to0$ and any convergent subsequence $x_{\alpha_{k_j}}\to\hat x$.
The cone is closed, so $\hat x\ge0$, and the first estimate drives the residual to zero, so $A\hat x=b$ and $\hat x\in M_0$.
Passing to the limit in the second estimate gives $\norm{\hat x-\bar x}\le\norm{x^*-\bar x}$, and uniqueness of the projection forces $\hat x=x^*$.
A bounded family all of whose accumulation points equal $x^*$ converges to $x^*$.
\end{proof}

\begin{proof}[Proof of \cref{thm:bregman}]
(a) The problem $\min_{Ax=b}\norm{x-\bar x}^2$ minimises a strictly convex objective under affine constraints, so its solution is characterised by Lagrangian stationarity.
Stationarity of $\norm{x-\bar x}^2-2\mu^{\!\top}(Ax-b)$ in $x$ gives $x-\bar x=A^{\!\top}\mu$.
Substituting into $Ax=b$ gives $AA^{\!\top}\mu=b-A\bar x$.
$A$ has full row rank by \cref{lem:rank}, so $AA^{\!\top}$ is invertible and $\mu=(AA^{\!\top})^{-1}(b-A\bar x)$.
Then $x=\bar x+A^{\!\top}\mu=\bar x-A^{\!\top}(AA^{\!\top})^{-1}(A\bar x-b)$, which is \cref{eq:proj}.
This $x$ is the orthogonal projection of $\bar x$ onto $M$, since the correction $x-\bar x$ lies in $\range A^{\!\top}=(\ker A)^{\perp}$.
For $\bar x=0$ it reduces to $x=A^{\!\top}(AA^{\!\top})^{-1}b$, the element of $M$ of least Euclidean norm, the normal solution.

(b) For the Kullback-Leibler divergence $D(x,\bar x)=\sum_j\bigl(x_j\log(x_j/\bar x_j)-x_j+\bar x_j\bigr)$ the value $D(x,\bar x)$ is finite exactly when the support of $x$ lies inside the support of $\bar x$, so the hypothesis $\bar x>0$ makes $D(\cdot,\bar x)$ finite on the whole feasible set.
The projection $\arg\min_{Ax=b,\,x\ge0}D(x,\bar x)$ is an $I$-projection onto the linear family cut out by $Ax=b$.
The feasible set $M\cap\real^m_+$ is compact and convex and $D(\cdot,\bar x)$ is lower semicontinuous and strictly convex, so the $I$-projection exists and is unique \citep{Csiszar1975}.
Csisz\'ar's condition $M\cap\relint\Delta\neq\varnothing$ plays a separate role: it yields the exponential-family form $x_j=\bar x_j\exp\bigl((A^{\!\top}\nu)_j\bigr)$ with finite multipliers $\nu$, while on the boundary the projection can exist with no finite multipliers.
The projection is computed by alternating Bregman projections onto the constraint blocks \citep{BregmanBCCNP2015}, which converge to the unique $I$-projection.
The scaling is weighted: the projection onto the balance constraint of target $j$ multiplies the coordinate of source $i$ by $v_j^{N^{(i)}}$, with the exponent set by the source mass, because the balance rows carry the masses as coefficients.
The diagonal Sinkhorn scaling corresponds to the divergence taken in the flow coordinates $z$, a different objective, since $D(z,\bar z)=\sum_i N^{(i)}D_i(x,\bar x)$ weights the sources by their masses and the divergence is not invariant under the diagonal change of coordinates.
For a uniform reference $\bar x$ the objective reduces to the negative entropy, so the projection is the maximum-entropy operator consistent with the data.
Both parts are the Bregman projection of $\bar x$ onto $M$ for the divergence $D$ \citep{Bregman1967}.
For the inertial reference $\bar x=P_0$ the hypothesis of (b) fails and no $I$-projection exists: the balance rows annihilate the retention columns, so every feasible $x$ with a nonzero increment carries mass outside the support of $P_0$ and $D(x,P_0)=+\infty$ on all of $M$.
\end{proof}


\begin{proof}[Proof of \cref{prop:bb}]
The Benamou-Brenier formulation reads optimal transport as the search for the curve of least kinetic energy between two marginals \citep{BenamouBrenier2000}.
We borrow from it the least-kinetic-energy functional alone, transplanted to the flat Frobenius geometry of the operator space, where no weighting density of the form $\rho\,|v|^2$ remains.
On the space of substochastic matrices with the Frobenius inner product, the kinetic energy of a curve $\tau\mapsto\Pt(\tau)$ is
\[
  \mathcal{E}[\Pt]=\int_0^{T}\norm{\dot\Pt(\tau)}^2\,d\tau .
\]
Sample the curve at the integer times and interpolate it linearly on each unit interval $[t-1,t]$.
The substochastic matrices form a convex set, so a convex combination of $\Pt(t-1)$ and $\Pt(t)$ stays substochastic and the interpolant is admissible.
On $[t-1,t]$ the interpolant has the constant velocity
\[
  \dot\Pt(\tau)=\frac{\Pt(t)-\Pt(t-1)}{\Delta\tau}=\Pt(t)-\Pt(t-1),\qquad \Delta\tau=1 .
\]
The integrand is constant on the interval, so the quadrature over it is exact,
\[
  \int_{t-1}^{t}\norm{\dot\Pt(\tau)}^2\,d\tau=\norm{\Pt(t)-\Pt(t-1)}^2 .
\]
Summing over the intervals gives $\mathcal{E}[\Pt]=\sum_t\norm{\Pt(t)-\Pt(t-1)}^2$, the discrete action of \eqref{eq:action} with $D$ the squared Frobenius distance and $\lambda_d=1$.
The interpolant also minimises the energy among all admissible curves through the same nodes.
For any curve $X$ on $[t-1,t]$ with $X(t-1)=\Pt(t-1)$ and $X(t)=\Pt(t)$, the Cauchy-Schwarz inequality on the unit interval gives
\[
  \int_{t-1}^{t}\norm{\dot X(\tau)}^2\,d\tau\ \ge\ \Bigl\lVert\int_{t-1}^{t}\dot X(\tau)\,d\tau\Bigr\rVert^2=\norm{\Pt(t)-\Pt(t-1)}^2,
\]
with equality for the constant velocity of the linear interpolant.
The dynamic term of the least-action prior is therefore the time-discretisation of the flat analogue of the Benamou-Brenier kinetic energy of the operator path, and its value is the least energy of any admissible path through the sampled operators.
\end{proof}

\begin{proof}[Derivation of \cref{rem:onsager}]
For the diffusion $dX=b(X)\,d\tau+\sqrt{2\varepsilon}\,dW$ the Onsager-Machlup functional of a smooth path on $[0,T]$ is $\int_0^T\bigl[\tfrac{1}{4\varepsilon}\norm{\dot X-b(X)}^2+\tfrac12(\nabla\!\cdot b)(X)\bigr]\,d\tau$ \citep{OnsagerMachlup1953}.
For the Ornstein-Uhlenbeck drift $b(X)=-\gamma(X-P_0)$ the divergence is the constant $-\gamma m$, with $m$ the dimension of the operator vector, so at a fixed horizon the correction shifts the action by a constant.
Expanding the square gives
\[
  \tfrac{1}{4\varepsilon}\norm{\dot X+\gamma(X-P_0)}^2
  =\tfrac{1}{4\varepsilon}\Bigl(\norm{\dot X}^2
  +\gamma\,\tfrac{d}{d\tau}\norm{X-P_0}^2
  +\gamma^2\norm{X-P_0}^2\Bigr),
\]
and the middle term integrates to the boundary value $\tfrac{\gamma}{4\varepsilon}\bigl(\norm{X(T)-P_0}^2-\norm{X(0)-P_0}^2\bigr)$, which involves the end operators alone.
The interior part of the action is therefore
\[
  \mathcal{S}[X]=\frac{1}{4\varepsilon}\int_0^T\bigl(\norm{\dot X}^2+\gamma^2\norm{X-P_0}^2\bigr)\,d\tau .
\]
The cross term carries the whole time asymmetry of the drift: once it moves to the boundary, the interior action is invariant under the reversal $\tau\mapsto T-\tau$, in line with the reversibility of the stationary Ornstein-Uhlenbeck process, and no directed drift term survives in the objective.
Sample the path on the grid of step $h$ and apply the rectangle rule with the difference quotient $\dot X\approx\bigl(\Pt(t)-\Pt(t-1)\bigr)/h$:
\[
  \mathcal{S}\ \approx\ \sum_t\Bigl[\frac{1}{4\varepsilon h}\norm{\Pt(t)-\Pt(t-1)}^2
  +\frac{\gamma^2h}{4\varepsilon}\norm{\Pt(t)-P_0}^2\Bigr],
\]
the two penalty terms of \eqref{eq:action} with $\lambda_d=1/(4\varepsilon h)$ and $\lambda_s=\gamma^2h/(4\varepsilon)$.
At $h=1$ the choice $\gamma^2=\lambda_s/\lambda_d$ and $\varepsilon=1/(4\lambda_d)$ realises every pair $\lambda_d,\lambda_s>0$.
In the discrete weights the boundary value of the middle term equals $\sqrt{\lambda_d\lambda_s}\,\bigl(\norm{\Pt(T)-P_0}^2-\norm{\Pt(0)-P_0}^2\bigr)$, where $\Pt(0)$ and $\Pt(T)$ are the first and last operators of the trajectory and $P_0$ is the fixed full-retention vector; the boundary value enters only the stationarity conditions at $t=0$ and $t=T$ and leaves the interior three-point recursion of \cref{rem:el} unchanged.
The quadrature reproduces the integral to leading order in $h$, and the exact one-step transition density of the reference departs from the separable two-term form at order $\gamma h$, so the identification holds at the level of the discretised action.
The discrete sum itself is a positive-definite quadratic form in the trajectory, hence, up to an additive constant, the negative logarithm of a Gaussian density, and its minimiser under the balance constraints is the mode of the conditioned Gaussian law; truncation to the cone $\Pt\ge0$ preserves the constrained minimiser, since restricting a log-concave density to a convex set does not move its constrained mode.
\end{proof}

\begin{proof}[Proof of \cref{prop:transport}]
Pass to the flow variables $z^{(i,j)}=N^{(i)}(t)\,\Pt^{(i,j)}(t)$ used for \cref{lem:kernel}, with the reserve entering through $N^{(0)}_{\mathrm{eff}}(t)$.
The masses are positive, so $\Pt\mapsto z$ is a diagonal change of coordinates, an affine isomorphism that carries the feasible set of one layer onto its image without altering its face structure.
In these coordinates the normalisation of source $i$ is the row-sum constraint $\sum_j z^{(i,j)}=N^{(i)}(t)$, and the balance of target $j$ combined with the normalisation of source $j$ as in \cref{lem:kernel} is the column-sum constraint $\sum_i z^{(i,j)}=N^{(j)}(t+1)$.
Departure $j=n+1$ carries no balance equation, so its column has no fixed sum and acts as a slack sink that absorbs the mass leaving the modelled clusters.
The set $\{z\ge 0:\ \text{row sums }N^{(i)}(t),\ \text{column sums }N^{(j)}(t+1),\ \text{free sink}\}$ is a transportation polytope with the observed marginals and one slack column, restricted to the admissible edges: the columns of the passive targets hold two admissible entries each, the reserve column one, and the diagonal change of coordinates leaves these structural zeros in place.
This proves the affine isomorphism with the stated qualifications.

The minimum-norm objective transforms without isometry.
For positive masses $\norm{\Pt}^2=\sum_{i,j}(z^{(i,j)})^2/(N^{(i)})^2$ identically, so the minimum-norm reference minimises the weighted quadratic form with source weights $1/(N^{(i)})^2$ on the polytope.
The Karush-Kuhn-Tucker conditions for this problem, with $\alpha_i$ the row multipliers, $\beta_j$ the column multipliers and $\beta_{n+1}=0$ at the slack sink, give $z^{(i,j)}=\max\bigl(a_i+(N^{(i)})^2 b_j,\,0\bigr)$ with $a_i=N^{(i)}\alpha_i/2$ and $b_j=\beta_j/2$ on the admissible edges.
The standard quadratically regularised plan minimises $\norm{z}^2$ on the same polytope and has the additive clipped potentials $z^{(i,j)}=\max(a_i+b_j,0)$ \citep{PeyreCuturi2019}.
When the source masses are equal, $N^{(i)}\equiv N$, the two objectives are proportional and the two plans coincide; for unequal masses the potential forms differ and so, in general, do the plans.

On the outflow simplex of a single source only the row-sum constraint acts.
The Cauchy-Schwarz inequality gives $\norm{\Pt^{(i,\cdot)}}^2\ge1/k$ over the $k$ admissible targets, with equality at the uniform outflow, and the Shannon entropy is maximised at the same point, so there the minimum-norm reference and the maximum-entropy plan coincide.
This coincidence is a property of the bare simplex and does not survive the coupling that the column constraints induce.
\end{proof}

\begin{proof}[Proof of \cref{lem:prox}]
The objective $(1-\lambda)\norm{x-\Pt(t)}^2+\lambda\norm{x-\Pt_s(t-1)}^2$ expands to $\norm{x-y}^2$ plus a constant, with $y=(1-\lambda)\Pt(t)+\lambda\,\Pt_s(t-1)$, so its unconstrained minimiser is $y$, which is temporal smoothing.
Rearranged, the recursion reads $\Pt_s(t)-\Pt_s(t-1)=(1-\lambda)\bigl(\Pt(t)-\Pt_s(t-1)\bigr)$, one explicit Euler step of size $h=1-\lambda$ of the relaxation
\[
  \dot X(\tau)=-\bigl(X(\tau)-\Pt(\tau)\bigr),
\]
the gradient flow of the quadratic potential $\tfrac12\norm{X-\Pt}^2$.
Under the constraint $\Amat(t)x=\Amat(t)\Pt(t)$ the same objective is minimised by the orthogonal projection of $y$ onto the fiber,
\[
  y-\Amat^{\!\top}(\Amat\Amat^{\!\top})^{-1}\bigl(\Amat y-\Amat\Pt(t)\bigr)
  =\Pt(t)+\lambda\bigl(I-\Amat^{\!\top}(\Amat\Amat^{\!\top})^{-1}\Amat\bigr)\bigl(\Pt_s(t-1)-\Pt(t)\bigr),
\]
which is the causal projection step realised by the recursion \eqref{eq:recursion} of \cref{subsec:nullspace}.
\end{proof}

\begin{proof}[Proof of \cref{lem:mass}]
$T$ is non-singular with respect to $m$, so its Perron-Frobenius transfer operator $L$ is defined by $\int_X (Lf)\,g\,dm=\int_X f\,(g\circ T)\,dm$ for all bounded measurable $g$.
Take $g\equiv 1$.
Then $\int_X Lf\,dm=\int_X f\,dm$, so $L$ conserves the mass of every density, which is the first claim \citep{Ulam1960,DellnitzJunge1999,Froyland2014}.

The cells partition the modelled domain $D$ and have positive mass, so Ulam's entries are defined, and the entry from source cell $i$ to target cell $j$ is the fraction of the mass of $B_i$ carried into $B_j$,
\[
  Q_{ji}=\frac{m\bigl(B_i\cap T^{-1}B_j\bigr)}{m(B_i)} .
\]
Column $i$ collects the destinations of $B_i$, so
\[
  \sum_j Q_{ji}=\frac{m\bigl(B_i\cap T^{-1}(\textstyle\bigcup_j B_j)\bigr)}{m(B_i)}=\frac{m\bigl(B_i\cap T^{-1}D\bigr)}{m(B_i)}\ \le\ 1 .
\]
The column sum is one exactly when $m$-almost all of $B_i$ lands in $D$ under $T$, and the deficit $1-\sum_j Q_{ji}$ is the share of $B_i$ that $T$ carries into $X\setminus D$, the leakage into the departure sink.
The paper keeps this orientation, sources down the columns, and treats departure as the only defect from column-stochasticity: on the cells that leak nothing, $\mathbf{1}^{\!\top}Q=\mathbf{1}^{\!\top}$.
\end{proof}

\begin{proof}[Proof of \cref{thm:coarse}]
In the physical instantiation of \cref{sec:physics} the ensemble follows the deterministic rapidity flow, so over a fixed lag $\Delta t$ the density is pushed forward by the transfer operator $P_{\Delta t}$ of the time-$\Delta t$ flow map $T$, a mass-conserving operator by \cref{lem:mass}; for a stochastic reference dynamics $P_{\Delta t}$ is the corresponding Markov transfer operator with the same mass conservation.

Project $P_{\Delta t}$ onto the cell indicators in the mass coordinates.
The projection sends a density $f$ to $\sum_j\bigl(\langle\mathbf{1}_{B_j},f\rangle_m/m(B_j)\bigr)\mathbf{1}_{B_j}$, and the coordinates of $\mathbf{1}_{B_i}/m(B_i)$ are the cell masses.
In these coordinates the matrix of the projected operator is Ulam's matrix,
\[
  Q_{ji}=\frac{\langle\mathbf{1}_{B_j},\,P_{\Delta t}\mathbf{1}_{B_i}\rangle_m}{\langle\mathbf{1}_{B_i},\mathbf{1}_{B_i}\rangle_m},
\]
the fraction of the mass initialised uniformly in cell $i$ that the flow carries into cell $j$ over the lag \citep{Ulam1960,Froyland2014}, and $Q$ is substochastic by \cref{lem:mass}.

Write the coarse density as the vector of cell masses $N_j(t)=\int_{B_j}\rho(\cdot,t)\,dm$, and read the transition fractions as the operator indicators $\Pt^{(i,j)}(t)=Q_{ji}$.
The mass balance of cell $j$ over one lag equates the increment of its mass to the net transported flow, and the outflow of each source cell closes to one through its retention share, with the departure sink carrying no balance equation of its own.
Stacking the per-cell balance rows and the per-source normalisation rows gives the linear system
\[
  \Amat(t)\,\Pt(t)=\ncal(t),\qquad \Pt(t)\ge 0,
\]
with $\Amat(t)$ set by the observed cell masses and $\ncal(t)$ by the increments $N(t+1)-N(t)$, the generic balance-plus-normalisation form of \cref{sec:vector}.
The kernel identification of \cref{lem:kernel} uses only this form, so $\ker\Amat$ is spanned by balanced $4$-cycles whenever every source cell can leak into departure and every target cell is reachable.
Recovering $Q$ from two coarse densities is the same segmentation problem, under-determined in the same way, since two coarse densities fix only the marginals of $Q$ and leave its internal routing on the solution manifold \citep{BiSarrazinSchmitzer2026,IoannouKlusDosReis2025}.
The split of the cells into active and passive clusters, the reserve and the two retention kinds is the modelling convention of \cref{sec:physics} and enters no step of this proof.
\end{proof}

\begin{proof}[Proof of \cref{lem:neutral}]
(i) For $\delta\in\ker\Amat(t)$ the update $\Pt(t)\mapsto\Pt(t)+\delta$ gives $\Amat(t)(\Pt(t)+\delta)=\Amat(t)\Pt(t)+\Amat(t)\delta=\Amat(t)\Pt(t)$, so the macro forecast is invariant.

(ii) $\Amat$ has full row rank by \cref{lem:rank}, so $\Amat\Amat^{\!\top}$ is invertible and the projection $\Pi_{M(t)}(x)=x-\Amat^{\!\top}(\Amat\Amat^{\!\top})^{-1}(\Amat x-b)$ of \cref{thm:bregman} is defined.
Applying $\Amat$ gives
\[
  \Amat\,\Pi_{M(t)}(x)=\Amat x-\Amat\Amat^{\!\top}(\Amat\Amat^{\!\top})^{-1}(\Amat x-b)=\Amat x-(\Amat x-b)=b,
\]
so the image lies in the fiber $M(t)$ and the balance holds exactly.

(iii) Let $\delta\in\ker\Amat(t-1)$ and set $\Delta A=\Amat(t)-\Amat(t-1)$.
Then $\Amat(t)\delta=\bigl(\Amat(t-1)+\Delta A\bigr)\delta=\Delta A\,\delta$, since $\Amat(t-1)\delta=0$.
The operator-norm inequality gives $\norm{\Amat(t)\delta}=\norm{\Delta A\,\delta}\le\norm{\Delta A}\,\norm{\delta}$.
The leakage into the forecast is governed by the drift $\Delta A$ of the system matrix and vanishes when $\Amat$ is time-invariant.
\end{proof}

\section{Full block structure of the balance matrix}\label{sec:appB}

This appendix records the explicit block form of the matrix $\Amat(t)$ whose size and role are summarised in \cref{sec:vector}.
The notation follows \cref{sec:model,sec:vector}.
Here $E_k$ is the $k\times k$ identity, $\Theta_{k_1,k_2}$ is the $k_1\times k_2$ zero matrix, and $\diag(\cdot)$ is a diagonal matrix.
A block written $A_{l_1,l_2}^{(\cdot)}$ has $l_1$ rows and $l_2$ columns.

\subsection{The observation vector}
The observation vector $\ncal(t,t+1)\in\real^{p}$, $p=4n+2$, stacks the balance increments over $(t,t+1)$ in its first $2n+1$ rows and a block of $2n+1$ ones in the remaining rows,
\begin{equation}\label{eq:appB-ncal}
\ncal(t,t+1)=\bigl(\Delta_1^1,\ldots,\Delta_1^n,\;\Delta_2^0,\;\Delta_2^1,\ldots,\Delta_2^n,\;\underbrace{1,\ldots,1}_{2n+1}\bigr)^{\mathrm{T}},
\end{equation}
with the increments
\[
\Delta_1^i=N_1^{(i)}(t+1)-N_1^{(i)}(t),\qquad
\Delta_2^i=N_2^{(i)}(t+1)-N_2^{(i)}(t),\qquad i=\overline{1,n},
\]
\[
\Delta_2^0=N_2^{(0)}(t+1)-N_2^{(0)}(t)-\Delta N_2^{(0)}(t).
\]
The reservoir increment $\Delta_2^0$ subtracts the exogenous inflow $\Delta N_2^{(0)}(t)$, so the balance rows measure the internal exchange alone.
The trailing ones are the right-hand sides of the $2n+1$ normalisations of \cref{sec:model}.

\subsection{The indicator vector}
The indicator vector $\Pt(t)\in\real^{m}$, $m=2n^2+5n+2$, lists the transition indicators in the canonical order $P_{11}\to P_{12}\to P_{21}\to P^{*}\to P^{**}$,
\begin{equation}\label{eq:appB-P}
\Pt(t)=\bigl(P_{11}(t),\;P_{12}(t),\;P_{21}(t),\;P^{*}(t),\;P^{**}(t)\bigr)^{\mathrm{T}}.
\end{equation}
The transition blocks expand as
\[
P_{11}(t)=\bigl(P_{11}^{(1)}(t),\ldots,P_{11}^{(n)}(t)\bigr),
\]
\[
P_{11}^{(i)}(t)=\bigl(P_{11}^{(i,1)}(t),\ldots,P_{11}^{(i,i-1)}(t),\;P_{11}^{(i,i+1)}(t),\ldots,P_{11}^{(i,n+1)}(t)\bigr),\qquad i=\overline{1,n},
\]
\[
P_{12}(t)=\bigl(P_{12}^{(1)}(t),\ldots,P_{12}^{(n)}(t)\bigr),
\]
\[
P_{21}(t)=\bigl(P_{21}^{(0)}(t),\;P_{21}^{(1)}(t),\ldots,P_{21}^{(n)}(t)\bigr),\qquad
P_{21}^{(i)}(t)=\bigl(P_{21}^{(i,1)}(t),\ldots,P_{21}^{(i,n+1)}(t)\bigr),\qquad i=\overline{0,n}.
\]
The retention blocks are
\[
P^{*}(t)=\bigl(P^{*(0)}(t),\;P^{*(1)}(t),\ldots,P^{*(n)}(t)\bigr),\qquad
P^{**}(t)=\bigl(P^{**(1)}(t),\ldots,P^{**(n)}(t)\bigr).
\]
The blocks $P_{11}$, $P_{12}$, $P_{21}$ carry $n^2$, $n$ and $(n+1)^2$ components, and the retention blocks $P^{*}$, $P^{**}$ carry $n+1$ and $n$ components, which gives $m=2n^2+5n+2$.

\subsection{The block form of the matrix}
The columns of $\Amat(t)$ follow the block order of $\Pt(t)$.
The first $2n+1$ rows form the balance block $B(t)$ and the remaining $2n+1$ rows form the normalisation block.
The retention indicators $P^{*}$ and $P^{**}$ cancel in the increments, so their columns vanish in $B(t)$.
They carry coefficient one in the normalisation block, where each normalisation of \cref{sec:model} closes the outgoing budget of one source to unity.
The balance block acts on the transition columns and has the $2\times 3$ super-block form
\begin{equation}\label{eq:appB-super}
B(t)=
\begin{pmatrix}
A_{n,\,n^2}^{(1)} & A_{n,\,n}^{(2)} & A_{n,\,(n+1)^2}^{(3)}\\[3pt]
\Theta_{n+1,\,n^2} & A_{n+1,\,n}^{(4)} & A_{n+1,\,(n+1)^2}^{(5)}
\end{pmatrix}.
\end{equation}
The column groups carry the $n^2$ indicators of $P_{11}$, the $n$ indicators of $P_{12}$ and the $(n+1)^2$ indicators of $P_{21}$.
The two row groups carry the $n$ active balance equations and the $n+1$ passive balance equations.
The zero block $\Theta_{n+1,\,n^2}$ records that moves inside the active clusters leave the passive balance equations untouched.

The block $A_{n,\,n^2}^{(1)}$ collects the coefficients of $P_{11}$ and splits into $n$ blocks of equal size,
\[
A_{n,\,n^2}^{(1)}=\bigl(A_{n,\,n}^{(1,1)},\,A_{n,\,n}^{(1,2)},\ldots,A_{n,\,n}^{(1,n)}\bigr),
\]
each of which stacks three sub-blocks,
\[
A_{n,\,n}^{(1,i)}=
\begin{pmatrix}
A_{i-1,\,n}^{(1,i,1)}\\[2pt]
A_{1,\,n}^{(1,i,2)}\\[2pt]
A_{n-i,\,n}^{(1,i,3)}
\end{pmatrix},\qquad i=\overline{1,n},
\]
with
\[
A_{i-1,\,n}^{(1,i,1)}=\bigl(N_1^{(i)}(t)\,E_{i-1},\;\Theta_{i-1,\,n-i+1}\bigr),
\]
\[
A_{1,\,n}^{(1,i,2)}=\bigl(\underbrace{-N_1^{(i)}(t),\,-N_1^{(i)}(t),\ldots,-N_1^{(i)}(t)}_{n}\bigr),
\]
\[
A_{n-i,\,n}^{(1,i,3)}=\bigl(\Theta_{n-i,\,i-1},\;N_1^{(i)}(t)\,E_{n-i},\;\Theta_{n-i,\,1}\bigr),\qquad i=\overline{1,n}.
\]

The block $A_{n,\,n}^{(2)}$ collects the coefficients of $P_{12}$ in the active balance equations,
\[
A_{n,\,n}^{(2)}=-\diag\bigl(N_1^{(1)}(t),\,N_1^{(2)}(t),\ldots,N_1^{(n)}(t)\bigr).
\]

The block $A_{n,\,(n+1)^2}^{(3)}$ collects the coefficients of $P_{21}$ in the active balance equations and splits into $n+1$ blocks,
\[
A_{n,\,(n+1)^2}^{(3)}=\bigl(A_{n,\,n+1}^{(3,0)},\,A_{n,\,n+1}^{(3,1)},\ldots,A_{n,\,n+1}^{(3,n)}\bigr),
\]
with
\[
A_{n,\,n+1}^{(3,0)}=\Bigl(\bigl(\Delta N_2^{(0)}(t)+N_2^{(0)}(t)\bigr)\,E_n,\;\Theta_{n,\,1}\Bigr),
\]
\[
A_{n,\,n+1}^{(3,i)}=\bigl(N_2^{(i)}(t)\,E_n,\;\Theta_{n,\,1}\bigr),\qquad i=\overline{1,n}.
\]
The leading block $A_{n,\,n+1}^{(3,0)}$ carries the reserve cluster, whose effective mass is the sum of the standing reserve $N_2^{(0)}(t)$ and the exogenous inflow $\Delta N_2^{(0)}(t)$.

The block $A_{n+1,\,n}^{(4)}$ collects the coefficients of $P_{12}$ in the passive balance equations,
\[
A_{n+1,\,n}^{(4)}=
\begin{pmatrix}
\Theta_{1,\,n}\\[2pt]
-A_{n,\,n}^{(2)}
\end{pmatrix}
=
\begin{pmatrix}
\Theta_{1,\,n}\\[2pt]
\diag\bigl(N_1^{(1)}(t),\ldots,N_1^{(n)}(t)\bigr)
\end{pmatrix}.
\]

The block $A_{n+1,\,(n+1)^2}^{(5)}$ collects the coefficients of $P_{21}$ in the passive balance equations and splits into $n+1$ blocks,
\[
A_{n+1,\,(n+1)^2}^{(5)}=\bigl(A_{n+1,\,n+1}^{(5,0)},\,A_{n+1,\,n+1}^{(5,1)},\ldots,A_{n+1,\,n+1}^{(5,n)}\bigr),
\]
each carrying a single nonzero row placed at the equation of its own cluster,
\[
A_{n+1,\,n+1}^{(5,i)}=
\begin{pmatrix}
\Theta_{i,\,n+1}\\[2pt]
\tilde A_{1,\,n+1}^{(5,i)}\\[2pt]
\Theta_{n-i,\,n+1}
\end{pmatrix},\qquad i=\overline{0,n},
\]
with
\[
\tilde A_{1,\,n+1}^{(5,0)}=\Bigl(\underbrace{-\bigl(\Delta N_2^{(0)}(t)+N_2^{(0)}(t)\bigr),\ldots,-\bigl(\Delta N_2^{(0)}(t)+N_2^{(0)}(t)\bigr)}_{n+1}\Bigr),
\]
\[
\tilde A_{1,\,n+1}^{(5,i)}=\bigl(\underbrace{-N_2^{(i)}(t),\,-N_2^{(i)}(t),\ldots,-N_2^{(i)}(t)}_{n+1}\bigr),\qquad i=\overline{1,n}.
\]

The departure columns $P_{11}^{(i,n+1)}$ and $P_{21}^{(i,n+1)}$ each keep a single nonzero entry in $B(t)$.
These entries give the balance block full row rank, as \cref{sec:nonident} uses.

\section{Concept briefs}\label{sec:appC}

This appendix recalls the four constructions that \cref{sec:action} draws on, at the level a
reader from another field needs.

\subsection{Optimal transport}\label{subsec:brief-ot}
Optimal transport moves one distribution onto another at least cost.
The Monge-Kantorovich problem seeks a coupling of two marginals that minimises an average
ground cost \citep{Villani2009,PeyreCuturi2019}.
For a quadratic cost the Benamou-Brenier form recasts this as a dynamic problem: among all
paths of distributions that join the two marginals, pick the one of least kinetic energy
\citep{BenamouBrenier2000}.
Adding an entropy term to the coupling gives entropic optimal transport, solved by the
Sinkhorn iterations, with the entropy weight interpolating between the least-cost coupling and
the maximum-entropy coupling \citep{Cuturi2013}.
The paper draws on that interpolation.

\subsection{The Schr\"odinger bridge}\label{subsec:brief-bridge}
The Schr\"odinger bridge asks for the most likely evolution between two observed marginals
under a reference stochastic dynamics, the law closest in relative entropy to the reference
among all laws with the given endpoints \citep{Leonard2014}.
It is the entropic regularisation of dynamic optimal transport, and a stochastic-control
reading connects the two through a Fisher-information term
\citep{ChenGeorgiouPavon2016,ChenGeorgiouPavon2021}.
The strong-friction limit of the bridge reduces the second-order dynamics to a first-order
relaxation.
\Cref{rem:bridge} relates that relaxation to the causal smoothing flow of \cref{lem:prox} and
records where the correspondence stops.

\subsection{The Perron-Frobenius operator and Ulam's method}\label{subsec:brief-pf}
The Perron-Frobenius operator pushes a density forward under a dynamical map and conserves
total mass by construction.
Ulam's method approximates it by a finite matrix: partition the state space into cells and
record the fraction of each cell that the map carries into every other cell
\citep{Ulam1960,DellnitzJunge1999}.
On a partition of the whole space the matrix is stochastic; on a modelled domain that the map
can leave it is substochastic, with the column deficit of a cell equal to the mass carried out
of the domain (\cref{lem:mass}).
Its leading eigenstructure approximates the invariant density and the slow modes of the
dynamics \citep{Froyland2014}.
The paper uses this matrix as the transfer operator of the phase-space cells.

\subsection{The Onsager-Machlup action}\label{subsec:brief-om}
The Onsager-Machlup action assigns to a trajectory of a fluctuating system a cost whose
minimiser is the most probable path \citep{OnsagerMachlup1953}.
For a gradient dynamics the action integrates the squared deviation from the deterministic
drift, and at a fixed horizon its interior part is the integral of kinetic energy plus
potential energy.
The discrete action on the operator trajectory in \cref{sec:action} is the time-discretisation
of this interior part for an Ornstein-Uhlenbeck reference drawn toward the full-retention
vector (\cref{rem:onsager}).

\section{Reproducibility}\label{sec:appD}

The code and data are archived on Zenodo \citep{NevecheryaCode2026}, under the concept
identifier 10.5281/zenodo.21140416, which always resolves to the latest release, and the
version identifier 10.5281/zenodo.21797159 for the release used here, version 1.3.1.
The archive is self-contained on numpy, scipy and matplotlib.
It holds the segmentation solver, the forecast pipeline, the physics simulation, the
labour-market inputs as plain JSON, the synthetic generators, the significance tests, the
applicability-map grids, and the figure generators, and every run in it is deterministic given
the fixed random seeds.
The solver settings, the regularisation weight, the convergence tolerance, the iteration count,
and the library versions are recorded there.
The labour-market panel is the curated OKVED-2 series for 2017 to 2024, and the rolling-origin
protocol gives the five holdouts 2020 to 2024.
The physics ensemble is integrated from the rapidity equation of \cref{sec:physics},
partitioned into the eight cells described there.

\bibliographystyle{plainnat}
\bibliography{references_arxiv}

\end{document}